\documentclass{amsart}
\usepackage{hyperref}
\usepackage{amsmath}
\usepackage{amsthm}
\usepackage{amssymb}
\usepackage{xfrac}
\usepackage{mathtools}
\usepackage{tikz-cd}
\usepackage[autostyle]{csquotes}
\usepackage{enumitem}
\usepackage{comment}
\usepackage{scalerel}
\usepackage{stackengine,wasysym}

\newtheorem{prop}{Proposition}

\theoremstyle{definition}

\newtheorem*{definition*}{Definition}
\newtheorem*{remark*}{Remark}

\newtheorem{remark}[prop]{Remark}

\newtheorem{question}{Question}

\theoremstyle{plain}

\newtheorem*{conj*}{Conjecture}
\newtheorem{theorem}[prop]{Theorem}
\newtheorem*{theorem*}{Theorem}
\newtheorem{thm}[prop]{Theorem}
\newtheorem{cor}[prop]{Corollary}
\newtheorem{lemma}[prop]{Lemma}

\newtheorem*{observation*}{Observation}

\numberwithin{prop}{section}
\numberwithin{equation}{section}

\newtheorem{Thm}{Theorem}

\newtheorem{Prop}[Thm]{Proposition}

\newtheorem{Cor}[Thm]{Corollary}

\newcommand{\NN}{\ensuremath{\mathbb{N}}}
\newcommand{\ZZ}{\ensuremath{\mathbb{Z}}}
\newcommand{\QQ}{\ensuremath{\mathbb{Q}}}
\newcommand{\CC}{\ensuremath{\mathbb{C}}}
\newcommand{\RR}{\ensuremath{\mathbb{R}}}

\newcommand{\TT}{\ensuremath{\mathbb{T}}}
\newcommand{\Ad}{\ensuremath{\mathrm{Ad}}}

\newcommand{\ab}{\ensuremath{\mathrm{ab}}}

\newcommand{\GL}[2]{\mathrm{GL}_{#1}\left(#2\right)}

\newcommand{\Fit}[1]{\mathrm{Fit}\left(#1\right)}
\newcommand{\vFit}[1]{\mathrm{vFit}\left(#1\right)}
\newcommand{\FC}[1]{\mathrm{FC}\left(#1\right)}

\newcommand{\chars}[1]{\mathrm{Ch}\left(#1\right)}
\newcommand{\traces}[1]{\mathrm{Tr}\left(#1\right)}
\newcommand{\tracesl}[1]{\mathrm{Tr}_{\le 1} \left(#1\right)}

\newcommand{\tracerel}[2]{\mathrm{Tr}_{#1}\left(#2\right)}

\newcommand{\Ch}[1]{\mathrm{Ch}\left(#1\right)}
\newcommand{\Tr}[1]{\mathrm{Tr}\left(#1\right)}

\newcommand{\relTr}[2]{\mathrm{Tr}_{#1}\left(#2\right)}
\newcommand{\tr}{\mathrm{tr}}
\newcommand{\Trfd}[1]{\mathrm{Tr_{fd}}\left(#1\right)}

\newcommand{\Prob}[1]{\mathrm{Prob}\left(#1\right)}

\newcommand{\Aut}[1]{\mathrm{Aut}\left(#1\right)}
\newcommand{\Ind}[3]{\mathrm{Ind}^{#1}_{#2}{#3}}
\newcommand{\normalizer}[2]{\mathrm{N}_{#1}\left(#2\right)}
\newcommand{\Sub}[1]{\mathrm{Sub}\left(#1\right)}
\newcommand{\End}[1]{\mathrm{End}\left(#1\right)}

\newcommand{\suptilde}[2]{{\widetilde{#1} {\Large\mathstrut}}^{#2}}

\newcommand{\folner}{{Følner}}

\global\long\def\Z{\mathbb{Z}}%
\global\long\def\R{\mathbb{R}}%
\global\long\def\C{\mathbb{C}}%
\title[Characters, stability and dense periodic measures]
{Characters of solvable groups, Hilbert--Schmidt stability and dense periodic measures}
\author{Arie Levit and Itamar Vigdorovich}

\begin{document}

\maketitle

\begin{abstract}
We study the character theory of metabelian and  polycyclic groups.  It is used to investigate Hilbert--Schmidt stability via the character-theoretic  criterion of Hadwin and Shulman. 
There is a close connection between stability and  dynamics of automorphisms of compact abelian groups. 
Relying on this,  we  deduce that finitely generated virtually nilpotent groups, free metabelian groups,  lamplighter groups as well as upper triangular groups over certain rings of algebraic integers  are Hilbert--Schmidt stable.
\end{abstract}

\section{Introduction}
 \label{sec:intro}

Characters play a key role in harmonic analysis on finite and on abelian groups.  Thoma  discovered a fruitful approach  that allows to study   characters of general infinite groups \cite{thoma1964unitare}.  
It   has  since been developed and   attracted considerable attention over the last two decades. 
\begin{definition*}
A \emph{trace} on a discrete group $G$ is a positive-definite, conjugation-invariant function $\varphi : G \to \CC$ normalized so that $\varphi(e) = 1$. A \emph{character} of $G$ is a trace which is not a   proper convex combination of traces. 
\end{definition*} 



 If a group is abelian then its character space  coincides with its Pontryagin dual. Other relevant examples of   characters   are  the normalized   traces $ \frac{1}{n}( \mathrm{tr} \circ \pi)$ where $\pi : G \to \mathrm{U}(n)$ is any   irreducible finite dimensional   unitary representation.

 %

%



A main part of this work is concerned with  developing  the character theory of solvable groups, focusing on the metabelian and polycyclic cases.  A recurring theme is \emph{monomiality}, namely  all characters are induced from abelian (or virtually abelian) subgroups. 
Prior to discussing this theory  in any greater detail, we mention the  notion of stability where  our results find  a perhaps unexpected application.

\subsection*{Stability and dense periodic measures}

A discrete group $G$ is called \emph{Hilbert--Schmidt stable} if  every \enquote{almost homomorphism} of $G$ into a finite-dimensional unitary group  is \enquote{nearby} an actual homomorphism with respect to the Hilbert--Schmidt metric.  The  formal definition  is outlined in \S\ref{sec:HS stability}.  For a survey see \cite{ioana2021almost}.

Hilbert--Schmidt stability
provides a potential direction in the   search for a non-hyperlinear group \cite{pestov2008hyperlinear,capraro2015introduction}. Quantitative versions of Hilbert--Schmidt stability play a crucial role in the proof of $\textrm{MIP}^*=\textrm{RE}$ and the refutation of the Connes embedding problem \cite{connes1976classification,ji2021mip, de2022spectral}.



In the realm of amenable groups there is an elegant and powerful criterion  for Hilbert--Schmidt stability that relies on characters.

\begin{theorem*}[Hadwin--Shulman \cite{hadwin2018stability}]
\label{thm:Hadwin}
An  amenable group $G$ is Hilbert--Schmidt stable if and only if every trace on $G$ is a pointwise limit of  normalized traces of finite dimensional unitary representations.
\end{theorem*}

Hadwin and Shulman use their criterion to show  that virtually abelian groups as well as the discrete Heisenberg group are Hilbert-Schmidt stable (see also  \cite{glebsky2010almost} for abelian groups).
We are not aware of   other Hilbert--Schmidt stability results for amenable groups in the literature, prior to this work. 

The problem of  determining whether this criterion holds  for a given amenable group $G$ turns out to be   intimately related to   topological dynamics. 
To illustrate, assume that  $G$ admits an abelian normal subgroup $N$. Any trace $\varphi $ on the group $G$ determines 
via the Bochner theorem
a $G$-invariant Borel probability measure $\mu_\varphi$ on the Pontryagin dual $\widehat{N}$. If the trace $\varphi$ is indeed a pointwise limit of  finite-dimensional traces then the measure $\mu_\varphi$ is a weak-$*$ limit of $G$-invariant probability measures of finite support. 
We are  led   to introduce the following notion. 


\begin{definition*}
A topological dynamical system $(G,X)$ has   \emph{dense periodic measures} if every $G$-invariant Borel probability measure on the compact space $X$ is a  weak-$*$ limit of $G$-invariant probability measures with finite supports.
\end{definition*}

For the most part we will  assume that the space $X$ is a compact abelian group and that the group $G$ is acting on $X$ by continuous automorphisms. The above discussion suggests a  relationship between Hilbert--Schmidt stability and the density of periodic measures for the action   on certain Pontryagin dual groups.

\begin{Prop}
\label{obs:intro: necc for HS}
Let $G$ be a  Hilbert--Schmidt stable  amenable   group.
Then for any     abelian normal subgroup $N$ of $G$ the dynamical system $(G,\widehat{N})$ has dense periodic measures.
\end{Prop}

To argue   from the density of periodic measures to stability, we require a better understanding of the  character theory of the group in question. 

\subsection*{Metabelian groups}
Our analysis     of solvable groups depends in an essential way on   \emph{induced characters}, a  notion   introduced and developed in \S\ref{sec:induced}.

\begin{Thm}
\label{thm:intro:characters of metabelian groups}
Let $G$ be a metabelian group. Then any character of  $G$ is induced from the abelianization $H^\ab$ of some normal subgroup $H \lhd G$.
\end{Thm}

Characters of particular metabelian groups were previously   classified in \cite{guichardet1963caracteres,bekka2020unitary}.  
We deduce the following stability criterion.

\begin{Thm}
\label{thm:intro:condition for metabelian to be stable}
Let $G$ be a finitely generated metabelian group. Assume that the topological dynamical system $(G,\widehat{H}^\ab)$ has dense periodic measures for any  normal subgroup $H \lhd G$.   Then the group $G$ is Hilbert--Schmidt stable\footnote{ Theorems \ref{thm:characters of metabelian groups} and \ref{thm:stability for metabelian} respectively are slightly more general formulations of Theorems \ref{thm:intro:characters of metabelian groups} and \ref{thm:intro:condition for metabelian to be stable}.}.
\end{Thm}

We find the question of determining which automorphisms  of compact abelian groups  have dense periodic measures    to be very natural and interesting. Unfortunately we don't know the   answer to this question in complete generality. 

Known  classes of  topological dynamical systems admitting dense periodic measures include   Bernoulli shifts \cite{parthasarathy1961category}  as well as torus automorphisms which are ergodic with respect to the Haar measure \cite{marcus1980note}. Loosely speaking, these systems enjoy a much stronger dynamical property called    \emph{periodic specification}  \cite{bowen1971periodic,ruelie1973statistical,lind1979ergodic,lind1982dynamical,lind1999homoclinic,kwietniak2016panorama}. 

\pagebreak

\begin{Cor} \label{Cor:metabelian-grps-HS-stable}
The following metabelian groups are Hilbert--Schmidt stable:
\begin{enumerate}
    
        \item Finitely generated free metabelian groups.
        \item Wreath products $A \wr \ZZ^d$ where $A$ is any finitely generated abelian group, e.g. the lamplighter groups.
\item  The Baumslag--Solitar  groups $\mathrm{BS}(1,n)$ for all non-zero $n\in \ZZ$.
    \item The semidirect products $\ZZ \ltimes_\alpha \ZZ^k$ where the automorphism $\alpha\in \GL{k}{\ZZ}$ is ergodic with respect to the Haar measure on the torus.  
\end{enumerate}
\end{Cor}


The property of stability for the class of finitely generated metabelian groups has a dynamical counterpart.


\begin{Thm}\label{Thm:metabelian-grp-charactarization-HS-stability}
The following two statements are equivalent:
\begin{enumerate}
    \item 
    \label{item:all metabelian are stable}
    All finitely generated  metabelian groups are Hilbert--Schmidt stable.
    \item 
    \label{item:all dcc are sofic}
    Any topological dynamical system $(G,X)$ where $G$ is a finitely generated abelian group acting on a compact abelian group  $X$ by automorphisms and satisfying the descending chain condition has dense periodic measures.
\end{enumerate}
\end{Thm}

We are unable to determine whether these  statements hold true in general. 
We remark that in statement (1)  it suffices to consider only \emph{split} metabelian groups. 
  
For topological dynamical systems   as the ones in Theorem   \ref{Thm:metabelian-grp-charactarization-HS-stability}  the descending chain condition (see p. \pageref{definition of dcc})   implies that periodic points are dense \cite[Theorem 5.7]{schmidt2012dynamical}.
  Foregoing the descending chain condition, it is not hard to find     dynamical systems without dense periodic measures. This leads to the following examples of infinitely generated\footnote{Finitely generated amenable non-residually finite  groups are never    Hilbert--Schmidt stable \cite[Proposition 2.5]{DogonAlon2021SaAo}, but this line of  reasoning does not apply to infinitely generated groups.} metabelian groups that are not Hilbert--Schmidt stable in light of    Proposition \ref{obs:intro: necc for HS}.

\begin{Cor}\label{Cor:compact-ring}
Let $k$ be a non-Archimedean local field with ring of integers $\mathcal{O}$. Let $A$ be an infinite   subgroup of  $\mathcal{O}^*$. Then the group $G = A \ltimes\widehat{\mathcal{O}}$ is not Hilbert--Schmidt stable.
 \end{Cor}
For example, the group $\ZZ\ltimes \ZZ(p^\infty)$ is not Hilbert--Schmidt stable  for any pair of distinct primes $p$ and $q$ where $\ZZ(p^\infty)$ is the Pr\"{u}fer $p$-group and the $\ZZ$-action corresponds to multiplication by powers of $q$.  This example demonstrates that Hilbert--Schmidt stability is not a local group property.



\subsection*{Polycyclic groups}


 Our main contribution  to the character theory of solvable groups is Theorem \ref{Thm:characters of virtually polycyclic groups}. This is, to the best of our knowledge, the first result dealing with characters of general polycyclic groups. 
 
 
We need to introduce a few notions first.   The \emph{FC-center} of a group $G$ is the characteristic subgroup $\FC{G}$ given by the union of the finite conjugacy classes. Given a virtually polycyclic group $G$ we let $\vFit{G}$ denote the maximal virtually nilpotent normal subgroup of $G$. The characteristic subgroup $\vFit{G}$ contains the Fitting subgroup $\Fit{G}$ with finite index (see \S\ref{sec:virtually polycycic} for further details).
Lastly, the kernel of a trace $\varphi$ is   $\ker \varphi = \varphi^{-1}(1)$. This is always a normal subgroup.




\begin{Thm}
\label{Thm:characters of virtually polycyclic groups}
Let $G$ be a virtually polycyclic group. For every  character $\varphi \in \Ch{G}$ 
there is a  finite index subgroup $H$ of $G$ and a quotient $\overline{H}$ of $H$ such that the character $\varphi$ is induced from the subquotient $\FC{\vFit{\overline{H}}}$.
\end{Thm}

This
 generalizes the classical theorem of 
 Howe \cite{howe1977representations}  --- any   character of a finitely generated nilpotent group $G$ is induced from  $\FC{G/\ker\varphi}$ \footnote{Works on characters of nilpotent groups include \cite{kaniuth1980ideals, carey1984characters,baggett1997primitive,tandra2004characters,kaniuth2006induced}.}.

Our proof of Theorem  \ref{Thm:characters of virtually polycyclic groups}
 relies on a   statement reminiscent of Mackey's theorem for induced characters that we find interesting in its own right (Theorem \ref{thm:Mackey intro}).
 
The leniency of being able to pass to  a finite index subgroup as well as considering 
$\mathrm{vFit}$ rather than $\mathrm{Fit}$ and the  FC-center rather than the  center is   necessary (even when dealing only with solvable groups).


In keeping with our philosophy and relying on   Theorem \ref {Thm:characters of virtually polycyclic groups} we get   the following    dynamical criterion for stability.






\begin{Thm}
\label{thm:intro-HS-stabilit-polycyclic-groups}
Let $G$ be a virtually polycyclic group. Assume that   any quotient $L$ of any finite index subgroup of  $G$  acts on the  Pontryagin dual of its subgroup $\mathrm{Z}( \vFit{L})$ 
with dense periodic measures.
Then the group $G$ is Hilbert--Schmidt stable. 
\end{Thm}

The action of any   group on the character space of its  center is trivial and therefore    has dense periodic orbits. We   immediately deduce:

\begin{Cor}
\label{Thm:nilpotent are HS}
Any finitely generated virtually nilpotent group  is Hilbert--Schmidt stable.
\end{Cor}

For any unipotent subgroup $G\leq \GL{k}{\ZZ}$ the semidirect product $G \ltimes \ZZ^k$ is a finitely generated nilpotent group. Therefore the dynamical system $(G,\TT^k)$ has dense periodic measures as a
consequence of Corollary \ref{Thm:nilpotent are HS} and of Proposition \ref{obs:intro: necc for HS} (the same  can be derived from Ratner's measure classification theorem \cite{ratner1991raghunathan}).

Recall that a finitely generated group $G$  has finite conjugacy classes    if and only if   $\left[G:\mathrm{Z}(G)\right] < \infty$ \cite[\S15.1]{scott2012group}. Together with  Theorem \ref{thm:intro-HS-stabilit-polycyclic-groups} this says that   any character of a virtually polycyclic group  is induced from some virtually abelian normal group. The dual topological dynamical system associated to this situation turns out to be topologically conjugate to an action  on a torus, see \S\ref{sec:fc induced}. Taking all this   into account allows us  to deal with a particular family of polycyclic groups.

\begin{Thm}
\label{Thm:upper triangular}
Let $\mathcal{O}$ be the ring of   algebraic integers of some number field $k$. Assume that the group of units $\mathcal{O}^*$ has rank one\footnote{Dirichlet's unit theorem says that $\mathrm{rank}( \mathcal{O}^*) = 1$   if and only if $k$ is either   a real quadratic field, a complex cubic field or a totally imaginary number field of degree $4$.}. Then the group of invertible upper triangular  matrices over the ring $\mathcal{O}$ is Hilbert--Schmidt stable.
\end{Thm}

Stability for the class of   virtually polycyclic groups turns out to be equivalent to 
 a dynamical property of torus automorphisms (note the analogy with Theorem  \ref{Thm:metabelian-grp-charactarization-HS-stability}).

\begin{Thm} \label{Thm:poly-grp-charactarization-HS-stability}
The following four statements are equivalent:
\begin{enumerate}
    \item \label{item:meta}
    All groups  of the form $\ZZ^d \ltimes \ZZ^k$ for some $d,k \in \NN$ are Hilbert--Schmidt stable.
    \item \label{item:vPoly} All virtually polycyclic groups\footnote{The class of virtually polycyclic groups coincides with the class of  finitely generated amenable subgroups of $\GL{k}{\ZZ}$.} are Hilbert--Schmidt stable.
    \item \label{item:abelian-sofic} The topological dynamical system $(G,\TT^k)$ has dense periodic measures for any \emph{abelian} subgroup $G \le    \GL{k}{\ZZ}$ and all $k \ge 2$.
    \item \label{item:amenable-sofic} The topological dynamical  system $(G,\TT^k)$ has dense periodic measures for any \emph{amenable} subgroup $G \le \GL{k}{\ZZ}$ and all $k \ge 2$.
\end{enumerate}
\end{Thm}

Surprisingly, to show that all virtually polycyclic groups are Hilbert--Schmidt stable it will be   enough to consider   the metabelian ones as in  statement (1). 

We do not know whether Statement (3)  of Theorem \ref{Thm:poly-grp-charactarization-HS-stability} is true. For a single toral automorphism this seems likely --- ergodic automorphisms are covered by the work of Marcus \cite{marcus1980note} and unipotent automorphisms were discussed  after Corollary \ref{Thm:nilpotent are HS}. The case of multiple commuting automorphisms seems to be much deeper. It is closely related to the higher rank measure rigidity conjecture   \cite{furstenberg1967disjointness}, \cite[Main Conjecture]{katok1996invariant}, \cite[Conjecture 3]{margulis2000problems},  \cite[Conjecture 4]{lindenstrauss2021recent}. 
Here is one explicit formulation of this conjecture.

\begin{conj*}[\cite{margulis2000problems}]
Let $\alpha$ be an  almost minimal  $\ZZ^d$-action on the   torus $\TT^k$, i.e. $\TT^k$ has no  proper closed infinite    $\ZZ^d$-invariant subsets.   Then  the only non-atomic  invariant probability measure on $\TT^k$ is the Haar measure.
\end{conj*}

Any $\ZZ^d$-action on the torus $\TT^k$ that satisfies the measure rigidity conjecture clearly has dense periodic measures. Here is an example of a  conditional application of the conjecture   to get stability.

 \begin{Prop}
\label{prop:application of conjecture to stability}
Assume that the measure rigidity conjecture holds true. Let $k$ be a totally real number field  with ring of algebraic integers $\mathcal{O}$. 
Then the metabelian polycyclic group $\mathcal{O}^* \ltimes \mathcal{O}$ is Hilbert--Schmidt stable.
\end{Prop}

Irreducible genuinely  partially hyperbolic  higher rank  actions   \cite[\S2.2.7]{katok2011rigidity}  are not covered by  the measure rigidity conjecture, i.e. they admit plenty non-atomic invariant probability measures singular with respect to the Haar measure on the torus. See Question \ref{quest:Katok's example} below for more details.

 \subsection*{Remarks}
 
\begin{enumerate}
\item  Not long after this work was announced, Eckhardt and Shulman \cite{eckhardt2023amenable} announced their work  which has a certain overlap with ours. For instance Eckhardt and Shulman also establish  Hilbert--Schmidt stability of finitely generated nilpotent groups, albeit in a very different way.
    \item The product of two Hilbert--Schmidt stable groups stays Hilbert--Schmidt stable provided one of the groups is amenable \cite[Corollary D]{ioana2021ii1}. So any  new example of  a Hilbert--Schmidt stable amenable group provides    non-amenable  examples.
        Works on Hilbert--Schmidt stability of non-amenable groups include \cite{atkinson2018some,hadwin2018stability,ioana2020cohomological,gerasimova2021virtually,ioana2021almost}.  See also  \cite{de2019operator,akhtiamov2022uniform} for a uniform version of Hilbert--Schmidt stability.
    \item Parallel to Hilbert--Schmidt stability is   \emph{stability in permutations} \cite{arzhantseva2015almost}, where the target groups are taken to be finite symmetric rather than unitary ones. This property    is related to sofic groups  \cite{gromov1999endomorphisms,weiss2000sofic,pestov2008hyperlinear,glebsky2009almost}. 
 The Hadwin--Shulman criterion has a close parallel in that setting --- a   finitely generated amenable group $G$ is stable in permutations if and only if any $G$-invariant Borel probability measure on its space of closed subgroups  is a weak-$*$ limit of $G$-invariant measures supported on finite index subgroups \cite{becker2019stability}. This dynamical criterion was used    to show that lamplighter   groups are stable in permutations \cite{levit2019infinitely}. All virtually polycyclic groups are  stable in permutations \cite[Corollary 8.2]{becker2019stability}. 
 \item 
It is natural to investigate the   connection between   permutation and Hilbert--Schmidt stability. The  preprint \cite{burton2021hyperlinear} shows that a  weakly permutation stable group must be \enquote{weakly Hilbert--Schmidt stable}. Recall that any finitely generated amenable group is weakly permutation stable in the sense of  \cite{arzhantseva2015almost}. To the best of our understanding, our Hilbert--Schmidt stability results    do not follow from the  methods of \cite{burton2021hyperlinear}. The recent paper  \cite{eckhardt2023amenable} includes a first example of an amenable (indeed solvable) group which is Hilbert--Schmidt stable but is not permutation stable  \cite[Theorem 5.23]{eckhardt2023amenable}
\item An incomplete list of works dealing with the character theory of other groups includes \cite{dudko2014finite,bekka2007operator,peterson2016character,thomas2018invariant,dudko2019invariant, bekka2020characters,lavi2020characters,bader2021charmenability,thomas2022characters}.  In most known cases all characters are  almost monomial, i.e are induced from virtually abelian subgroups.
\end{enumerate}

\subsection*{Open questions}

Here are   some  open problems suggested by our work.

\begin{question}
\label{quest:going up finite index}
Let $H$ be a finite index subgroup of $G$. Assume that the system $(G,X)$ has dense periodic measures. Is the same   true for the system $(H,X)$?
\end{question}

The converse direction to Question \ref{quest:going up finite index} is true, see Lemma  \ref{lem:DPM-finite-index} below.

\begin{question}
\label{quest:any torus auto has dense periodic measures}
Do all  actions  on the torus generated by a single automorphism admit   dense periodic measures?
\end{question}

The answer to Question \ref{quest:any torus auto has dense periodic measures}  is positive if the action is totally irreducible. The remaining case is essentially that of reducible transformations of  \enquote{mixed type} (i.e. hyperbolic and unipotent). We find the following problem concerning genuinely partially hyperbolic higher rank $\ZZ^d$-action   to be particularly intriguing.

\begin{question}
\label{quest:Katok's example}
Let $A,B \in \mathrm{GL}_6(\ZZ)$ be the pair of partially hyperbolic commuting matrices from   \cite[Example 2.2.20]{katok2011rigidity}. The corresponding dynamical system $(\ZZ^2,\TT^6)$ is \emph{not} almost minimal.   Does it admit dense periodic measures?
\end{question}

\begin{question}
Consider the finitely generated  metabelian group  $\ZZ^2\ltimes \Z[1/6]$ closely related to Furstenberg's famous   $\times 2,\times 3$ problem. Is it Hilbert--Schmidt stable? Does the   topological dynamical system $ (\ZZ^2, \widehat{\ZZ\left[1/6\right]})$ admit dense periodic measures?
\end{question}

A vast generalization of Questions $2,3$ and $4$ is the following.
\begin{question}
\label{quest:non-sofic}
Are all finitely generated metabelian groups Hilbert--Schmidt stable? Equivalently, do all systems $(\ZZ^d,X)$ satisfying the descending chain condition have dense periodic measures?
\end{question}

 The authors intend to address the following problem in a forthcoming work.
\begin{question}
Is every solvable minimax   group  almost monomial?
\end{question}
\subsection*{Acknowledgements} The authors would like to thank Uri Bader, Alon Dogon, Alex Furman, Elon Lindenstrauss, Alex Lubotzky and  Barak Weiss as well as the anonymous referee for helpful comments  and suggestions.


\setcounter{tocdepth}{1}
\begin{samepage}
\tableofcontents
\end{samepage}

\section{Characters and trace representations} \label{sec:characters}

 We define traces and characters and discuss their relation to trace representations via von Neumann algebras. We  introduce the notions of relative and dominated traces.
Let $G$ be a countable  group. We use the notation $g^h = h^{-1} g h$ for any pair of elements $g,h \in G$. This makes  conjugation  a right action.

\begin{definition*}
A \emph{trace} on the group $G$ is a function $\varphi:G\to \CC$ so that  
\begin{enumerate}
    \item $\varphi$ is \emph{positive definite}, i.e. \begin{equation*}
    \sum_{i=1}^n \sum_{j=1}^n\overline{\alpha}_i\alpha_j\varphi(g_i^{-1}g_j)\geq 0    
    \end{equation*} 
    for all $n\in \NN$ and any choice of  elements $g_1,...,g_n\in G$ and $\alpha_1,...,\alpha_n\in \CC$,
    \item $\varphi$ is \emph{conjugation-invariant}, i.e. $\varphi(g^h)=\varphi(g)$ for all   elements $g,h\in G$, and
    \item $\varphi$ is \emph{normalized}, i.e.  $\varphi(e)=1$.
\end{enumerate}
\end{definition*}
The \emph{kernel} of a trace $\varphi$ on the group $G$ is the normal subgroup $\ker\varphi=\varphi^{-1}(1)$. We say that the trace $\varphi$ is \emph{faithful} if the subgroup $\ker\varphi$ is trivial.  Any trace $\varphi$ on the group $G$ descends to a faithful trace on the quotient group $G/\ker\varphi$ \cite[Lemma 12.A.1]{bekka2020unitary}.

Let $\Tr{G}$ denote the space of traces on the group $G$ endowed with    the topology of pointwise convergence. This makes $\Tr{G}$ a compact  convex subset of $\ell^\infty(G)$. 
Let $\chars{G}$ denote the set of extreme points of $\traces{G}$.   The members of $\chars{G}$ are called \emph{characters}\footnote{Some authors refer to traces as \enquote{characters} and to extreme traces as \enquote{indecomposable characters}.}. 

It is known that   $\Tr{G}$ is a metrizable Choquet simplex \cite{thoma1964unitare}. This means that  the barycenter map
\begin{equation}\label{eq:fourier}
\Prob{\Ch{G}}\to \Tr{G}, \quad \mu  \mapsto \varphi = \int_{\Ch{G}} \psi \, \mathrm{d} \mu(\psi)
\end{equation}
is a continuous affine bijection. We refer to this map as the $\emph{Fourier transform}$. See \cite{phelps2001lectures} for an overview of Choquet theory.  

If the discrete group $G$ is  abelian then any character of $G$ is multiplicative (see Lemma \ref{obs:restriction is multiplicative}). 
Therefore the character space  $\Ch{G}$ coincides with the compact Pontryagin dual group $\widehat{G}$. The fact that the barycenter map in Equation (\ref{eq:fourier}) is a bijection in this case is known as Bochner's theorem \cite{folland2016course}. As the space  $\Prob{\Ch{G}}$ is compact  the barycenter map is a homeomorphism.

\subsection*{Trace representations}

Von Neumann algebras provide a way to generalize the intimate relationship between  characters and representations    from finite to   infinite groups.

Let $M$ be a von Neumann algebra \cite{dixmier2011neumann}. 
A \emph{trace} on the von Neumann algebra $M$  is a positive linear functional $\tau:M\to \C$ satisfying $\tau(1_M)=1$ as well as  $\tau(xy)=\tau(yx)$ for all elements $x,y\in M$. The trace $\tau$ is called \emph{faithful} if the condition $\tau(x^*x)=0$ implies $x=0$ for all elements $x \in M$. The trace $\tau$ is called    \emph{normal} if $\tau$ is continuous in the ultraweak topology. 

A von Neumann algebra $M$ is called  a \emph{factor} if its   center $\mathrm{Z}(M)$  consists  of scalar multiples of $1_M$. A \emph{finite factor} is a factor von Neumann algebra admitting a trace. A classical result by Murray and von Neumann states that the trace on a finite factor is unique, normal and faithful  \cite{murray1937rings}. 
For example, the finite dimensional von Neumann algebra  $\mathcal{M}_d(\C)$ of all $d$-by-$d$ complex matrices is a finite factor. Its  unique trace is the standard normalized trace $\tr_d:A\mapsto \frac{1}{d}\sum_{i=1}^d A_{ii}$.

\begin{definition*}
A \emph{trace representation} of the group $G$ is a triple $(M,\pi,\tau)$
where
\begin{enumerate}
\item $M$ is a von Neumann algebra,
\item $\pi : G \to \mathcal{U}(M)$ is a representation of the group $G$ into the group $\mathcal{U}(M)$ of unitaries of the algebra $M$
such that the image $\pi(G)$ generates $M$ as a von Neumann algebra, and
\item $\tau$ is a normal   faithful trace on the von Neumann algebra $M$.
\end{enumerate}
\end{definition*}

Saying that the image $\pi(G) $ generates the von Neumann algebra $M$  is equivalent to requiring that $\pi(G)''=M$ by the von Neumann double commutant theorem \cite[Theorem 1.2.1]{arveson2012invitation}.

An isomorphism between  two trace  representations $(M_{1},\pi_{1},\tau_{1})$ and $  (M_{2},\pi_{2},\tau_{2})$ of the group $G$ is an isomorphism of von Neumann algebras $f:M_{1}\to M_{2}$ satisfying
 $f\circ\pi_{1}=\pi_{2}$ and $\tau_{1}=\tau_{2}\circ f$.

Any trace representation $(M,\pi,\tau)$ gives rise to the trace $\tau\circ\pi\in \Tr{G}$. The GNS construction can be used to obtain the following fundamental correspondence.  For the proof see e.g. \cite[Theorem 5.1.10]{peterson2013notes}.

\begin{thm}[Thoma's correspondence]	\label{thm:thoma correspondence }   
The set of traces  $\Tr{G}$ on the group $G$ stands in a bijective correspondence with the set of isomorphism classes of trace representations of $G$. Moreover
\begin{itemize}
    \item the trace representation $(M,\pi,\tau)$   corresponds to the trace $\varphi = \tau \circ \pi$, and
    \item $\varphi$ is a character if and only if the 
    von Neumann algebra $M$ is a factor.
\end{itemize}
\end{thm}



\subsection*{Dominated traces}
There is a natural  partial ordering on the set of traces on the group $G$. To define this ordering it will be useful  to consider the compact convex set   $\tracesl{G}$ consisting of all positive definite conjugation-invariant functions $\varphi :G \to \CC $ satisfying   $\varphi(e)\leq 1$. 

Given a pair of traces\footnote{We abuse our standard terminology and refer to the elements of $\tracesl{G}$ as traces, even through these functions need not be normalized.}  $\varphi_1,\varphi_2\in \tracesl{G}$ we say that $\varphi_1$ is \emph{dominated} by $\varphi_2$ if the function $\varphi_2 - \varphi_1$ is positive definite. This condition is equivalent to saying that $\varphi_2 - \varphi_1 \in \tracesl{G}$. We denote this relation by $\varphi_1 \le \varphi_2$.
The two   traces
$\varphi_1,  \varphi_2 \in \traces{G}$ are  \emph{disjoint} if the only  trace $\psi \in \tracesl{G}$ satisfying  both $\psi\le \varphi_1$ and $\psi \le \varphi_2$ is the zero trace $\psi = 0$.

Let $(M,\pi,\tau)$ be a trace representation of the group $G$ with corresponding trace $\varphi = \tau\circ\pi$. For every  positive central element $T\in \mathrm{Z}(M)$ consider the function $\varphi_T$ on the group $G$ given by
\begin{equation}
\label{eq:phi_T definition}
\varphi_{T} =\tau\circ \Ad_{T^{\sfrac{1}{2}}} \circ \pi
\end{equation}
where
$\Ad:M\to \End{M}$ is the \emph{adjoint representation} defined by 
\begin{equation}\label{eq:adjoint}
\Ad_x(y)=x y x^* \quad \forall x,y\in M.
\end{equation}
    The mapping $T \mapsto \varphi_T$ sets up a bijective correspondence between the set of central elements $T\in Z(M)$ satisfying $0_M\leq T\leq 1_M$ and the set of traces in $ \tracesl{G}$ dominated by the given trace $\varphi$ \cite[Lemma 11.C.3]{bekka2020unitary}.

\subsection*{Relative traces}

Let $N$ be a normal subgroup of the group $G$.   The group $G$ admits a left action on the set of traces  $\tracesl{N}$ via  conjugation $\varphi\mapsto \varphi^g$ where 
\begin{equation}
 \varphi^g(n)=\varphi(n^g ) \quad    \forall g \in G, \, n \in N.
\end{equation}
The conjugation action is by affine homeomorphisms. It preserves the partial order relation of dominated traces.

Consider the compact   convex subset of $\Tr{N}$ consisting of traces on $N$ fixed by    the conjugation action of $G$, namely
\begin{equation}
\tracerel{G}{N}=\{\varphi\in\traces{N}\, : \, \varphi^g=\varphi \quad \forall g \in G \}.
\end{equation}
The members of $\tracerel{G}{N}$ are called \emph{relative traces}. For every trace $\varphi \in \traces{G}$ the    restriction $\varphi_{|N}$    is a relative trace in $\tracerel{G}{N}$. The Fourier transform of the restriction $\varphi_{|N} \in \Tr{N}$ determines  a $G$-invariant Borel probability measure $\mu_\varphi \in \mathrm{Prob} (\chars{N})$ such that $\varphi_{|N} = \int \psi \, \mathrm{d} \mu_\varphi(\psi)$. The probability measure $\mu_\varphi$ is ergodic if and only if the trace $\varphi$ is a character.



\begin{lemma}\label{lem:subtraces-projections}
Let $(M,\pi,\tau)$ be a trace representation of the group $G$   corresponding to the trace $\varphi=\tau\circ\pi$.
Let $N$ be a normal subgroup of $G$. Consider the relative trace $\psi=\varphi_{|N} \in \relTr{G}{N}$.  Then
\begin{enumerate}
    \item  The von Neumann subalgebra $Q =\pi(N)''\leq M$ is $G$-invariant.
    \item $(Q,\pi_{|N},\tau_{|Q})$ is the trace representation corresponding to the trace $\psi$.
    \item The mapping 
        \[
            T \mapsto \psi_T = \tau\circ \Ad_{T^{1/2}}\circ\pi_{|N}
        \] 
        is a $G$-equivariant order-preserving bijection from the set of    central elements $T \in \mathrm{Z}(Q)$ satisfying $0\leq T\leq 1_Q$  to the set of traces   in $ \tracesl N$ dominated by the relative trace $\psi$. 
    \item The central element  $T$ is a projection if and only if the two traces $\psi_T$ and $\psi_{1_Q-T}=\psi-\psi_T$ are disjoint.
\end{enumerate}
\end{lemma}

Throughout   Lemma \ref{lem:subtraces-projections} we regard the von Neumann algebra $M$ with the $G$-action given by the adjoint representation defined in   Equation (\ref{eq:adjoint}). In other words   $x^g  = \Ad_{\pi(g)}x$ for all $x \in M$ and all elements $g \in G$.

\begin{proof}[Proof of Lemma \ref{lem:subtraces-projections}]
Statements (1) and (2) are immediate. The fact that the mapping  $T\mapsto \psi_T$ is  a   bijection between the two sets in question has been discussed after Equation (\ref{eq:adjoint}).
 This mapping is $G$-equivariant as
\begin{align}  \begin{split}
    \varphi_{T^g}(n)&=(\tau\circ\mathrm{Ad}_{(T^g)^{\sfrac{1}{2}}} \circ \pi)(n) =
    \tau(T^g \pi(n)) = 
    \tau\left(\pi(g)T\pi(g)^{-1}\pi(n)\right) 
    = \\
    &= \tau\left(T\pi(g^{-1}ng)\right) = (\tau \circ \mathrm{Ad}_{T^{\sfrac{1}{2}}} \circ \pi)(n^g) = 
     \psi_T(n^g)=\psi_T^g(n)
    \end{split}
\end{align}
 holds true for all elements $g\in G$ and $n\in N$.  This gives statement  (3). 

For statement (4) observe that the bijection $T \leftrightarrow \psi_T$ is additive in the sense that  $\psi_{T_1+T_2}=\psi_{T_1}+\psi_{T_2}$ holds true for any pair of central elements $T_1, T_2 \in \mathrm{Z}(Q)$. Moreover $\psi_{1_Q}=\psi$ and  $\psi_{0_Q}=0$. With this in mind  it is clear that the bijection $T \leftrightarrow \psi_T$   is order-preserving. Therefore  the two traces $\psi_{T_1}$ and $\psi_{T_2}$ are disjoint if and only if the two elements $T_1$ and $T_2$ are disjoint in the sense that $T_1  T_2=0$.  It follows that  the two traces $\psi_T$ and $\psi_{I-T}$ are disjoint   if and only if $T(I-T)=0$. This last condition is equivalent to $T$ being a projection. 
\end{proof}

\section{Induction of characters}
\label{sec:induced}


 We discuss induction of traces and characters.  Let $G$ be a discrete group. For the purpose of the following discussion  fix a subgroup $H$ of $G$.  

\begin{definition*}\label{def:induction}
A trace $\varphi $ on the subgroup $H$ is  \emph{almost $G$-invariant} if
the subgroup 
\begin{equation}
\label{eq:G phi}
    G_ \varphi =\left\{ g\in \normalizer{G}{H}:\varphi^{g}=\varphi\right\}
\end{equation}
has finite index in  $G$.
\end{definition*}

Almost $G$-invariant traces can be induced from the subgroup $H$ to the group $G$. Let us explain how this is done.
The \emph{trivial extension} of any function $\varphi : H \to \CC$ is the function $\widetilde{\varphi} : G \to \CC$ defined by
\begin{equation}
\label{eq:extended trace}
 \widetilde{\varphi}(g) = 
 \begin{cases} \varphi(g) & g \in H \\ 0 & g \notin H.
 \end{cases}
\end{equation}
If $\varphi$ is a positive definite function on the subgroup $H$ then $\widetilde{\varphi}$ is a positive definite function on the group $G$, see e.g. \cite[Proposition 1.F.10]{bekka2020unitary}. In particular if the subgroup $H$ is   normal  and $\varphi \in \tracerel{G}{H}$ then $\widetilde{\varphi} \in \Tr{G}$. The general induction procedure is  as follows. A very similar definition of induced characters is studied in \cite{kaniuth2006induced}.






 \begin{definition*}
 Let $\varphi \in \Tr{H}$ be an almost $G$-invariant trace on the subgroup $H$. The \emph{induced trace} $\Ind{G}{H}{\varphi} \in \traces{G}$ is given by
\begin{equation}
\label{eq:induced trace}
\Ind{G}{H}{\varphi}=\frac{1}{[G:G_\varphi ]}\sum_{g \in G/G_\varphi } \suptilde{\varphi}{g}.
\end{equation}
\end{definition*}

Note that $\varphi \mapsto \varphi^g$ defines a left action.
 The expression on the right-hand side of   Equation (\ref{eq:induced trace}) is independent of the particular choice of the coset representatives for $G/G_\varphi$. The formula in  Equation (\ref{eq:induced trace}) continues to hold true if the subgroup $G_\varphi$ is replaced by any other finite  index subgroup $L$ of $G$ contained in  $   G_\varphi$. 

The function $\Ind{G}{H}{\varphi}$ is normalized and positive definite since it is a convex combination of finitely many normalized positive definite
functions of the form $\widetilde{\varphi}^g$. This function is  conjugation-invariant as 
\begin{equation}
\left(\Ind{G}{H}{\varphi}\right)^{g'}=\frac{1}{[G:G_\varphi]}\sum_{g\in G/G_\varphi}(\suptilde{\varphi}{g}){\Large\mathstrut}^{g'}=\frac{1}{[G:G_\varphi]}\sum_{g\in G/G_\varphi} \suptilde{\varphi}{g'g} =\Ind GH\varphi
\end{equation}
for any element $g' \in G$. We conclude that $\Ind{G}{H}{\varphi}$ is indeed a trace on $G$.

This notion of induced traces generalizes the following situations:
\begin{itemize}
    \item  If $H$ is a normal subgroup of $G$ and $\varphi \in \relTr{G}{H}$ is a  relative trace,  then $\Ind{G}{H} \varphi = \widetilde{\varphi}$.
    \item  If $H$ is a finite index subgroup of $G$ then any trace  $\varphi \in \traces{H}$ is almost $G$-invariant  and
\begin{equation*}
 \Ind{G}{H} \varphi = \frac{1}{\left[G:H\right]} \sum_{g \in G/H} \suptilde{\varphi}{g}.
\end{equation*}
    \item Generalizing the two previous examples, if $\left[G:\mathrm{N}_G(H) \right] < \infty$ (i.e. $H$ is an almost normal  subgroup of $G$)  then any trace  $\varphi \in \relTr{\mathrm{N}_G(H)}{H}$ is almost $G$-invariant and 
\begin{equation*}
    \Ind{G}{H} \varphi = \frac{1}{\left[G:\mathrm{N}_G(H)\right]} \sum_{g \in G/\mathrm{N}_G(H)} \suptilde{\varphi}{g}.
\end{equation*}
\end{itemize}
 
The induced trace has several very useful properties.

\begin{lemma}[Induction in stages]
\label{lem:induction in stages}
Let $H$ and $L$ be a pair of subgroups of $G$ satisfying $H \le L \le G$ and $\left[G:L\right] < \infty$. Let $\varphi\in\Tr H$ be an almost $G$-invariant trace. Then 
\begin{equation}
\label{eq:induced in stages}
\Ind{G}{L}{\, \Ind {L}{H}\varphi}=\Ind{G}{H}{\varphi}.
\end{equation}
\end{lemma}

\begin{proof}
To begin with note that $\varphi$ is an almost $L$-invariant trace and $\Ind{L}{H} \varphi$ is an almost $G$-invariant trace so that the left-hand side of Equation (\ref{eq:induced in stages}) is well defined. Write
\begin{align}
    \begin{split}
\Ind{G}{L}{\left(\Ind{L}{H}{\varphi}\right)} &=
\frac{1}{[G:L]}\sum_{g\in G/L}\left(\widetilde{\Ind{L}{H}{\varphi}}\right)^{g}= \\
&= 
\frac{1}{\left[G:L\right]}\frac{1}{[L:L_\varphi]}\sum_{g\in G/L}\sum_{l\in L/L_\varphi}\suptilde{\varphi}{gl}=\\
&= \frac{1}{[G:L_\varphi]}\sum_{h\in G/L_\varphi}\suptilde{\varphi}{h}=\Ind{G}{H}{\varphi}.
\qedhere
\end{split}
\end{align}
\end{proof}

We remark that Lemma \ref{lem:induction in stages} remains true with a similar proof, if the assumption that the subgroup $L$ has finite index in $G$   is replaced by the more general assumption   that the trivial extension of $\varphi$ to $L$ is almost $G$-invariant. We will not have occasion to use   this more general statement. 

\begin{lemma}[Continuity of induction]\label{lem:continuity of induction}
Let $\varphi_n \in \traces{H}$ be a sequence of traces on the subgroup $H$ converging to the trace $\varphi \in \traces{H}$ in the pointwise topology. Assume that the subgroup
\begin{equation}
    L = \bigcap_{n\in \NN} G_{\varphi_n}
\end{equation} 
has finite index in $G$.  Then the sequence of traces $\Ind{G}{H}{\varphi_n}$ converges to $\Ind{G}{H}{\varphi}$ in the pointwise topology.
\end{lemma}

\begin{proof}
The assumption implies that the traces $\varphi_n$ as well as the trace $\varphi$ are all almost $G$-invariant. It is clear that the sequence of positive definite functions $\widetilde{\varphi}_n^{g}$ converges to the positive definite function $\suptilde{\varphi}{g}$ for any fixed element $g\in G$. Therefore
\begin{equation}
    \Ind{G}{H}{\varphi_n}= \frac{1}{[G:L]}\sum_{g\in G/L} \widetilde{\varphi}_n^{g} \xrightarrow[n\to \infty]{} \frac{1}{[G:L]}\sum_{g\in G/L} \widetilde{\varphi}^{g}=\Ind{G}{H}{\varphi}
\end{equation}
in the pointwise topology.
\end{proof}

\subsection*{Mackey theory}
We introduce the following Mackey-type criterion which helps to identify when a given character is induced from a finite index subgroup.  
\begin{theorem}
\label{thm:Mackey intro}
 Let $N$ be a normal subgroup of the discrete group $G$. Let $\varphi$ be a character of $G$. Assume that
 \begin{equation}
\varphi_{|N} = \frac{1}{n} (\psi_1 + \cdots + \psi_n)
 \end{equation}
 for some $n \in \NN$ and some family of pairwise disjoint traces $\psi_i$ of the group $N$ permuted transitively by the $G$-action. Then each trace $\psi_i$ can be extended to a character $\varphi_i$ of the isotropy group $G_i = \mathrm{stab}_G (\psi_i)$ so that
 \begin{equation}\varphi = \Ind{G}{G_i}{\varphi_i}.
 \end{equation}
 \end{theorem}

\begin{proof}
 Let $\varphi \in \chars{G}$ be a character with a corresponding trace representation $(M,\pi,\tau)$. Here $M$ is a finite factor von Neumann algebra, $\pi$ is a representation of the group $G$ into the group $\mathcal{U}(M)$ of unitaries of  $M$ and $\tau$ is a normal faithful trace on $M$ satisfying   $\varphi = \tau \circ \pi$.
 
Let $N$ be a normal subgroup of the group $G$. Assume that  the restriction $\varphi_{|N}$ of the character $\varphi$ can be written as
 \begin{equation}
\varphi_{|N} = \frac{1}{n} (\psi_1 + \cdots + \psi_n)
 \end{equation}
 for some family of \emph{pairwise disjoint} traces $\psi_i$ of the group $N$ that are permuted transitively by the natural $G$-action. Denote $Q = \pi(N)''$ so that $Q$ is a von Neumann subalgebra of $M$ with $1_M = 1_Q \in Q$. Moreover denote $G_1 = \mathrm{stab}_G(\psi_1)$. Let $t_i \in G$ be a choice of coset representatives for $G_1$ in $G$ so that $\psi_1^{t_i} = \psi_i$ for all $i$. We may assume without loss of generality that $t_1 = e_G$.
  
There exist   central  orthogonal projections $p_{1},...,p_{n}\in \mathrm{Z}(Q)$
satisfying 
\begin{equation}
\sum_{i=1}^n p_{i}=1_{Q}=1_{M} \quad \text{and} \quad
p_i p_j = 0 \quad \forall i \neq j
\end{equation}
as well as
\begin{equation}
\frac{1}{n} \psi_i=(\varphi_{| N})_{p_i} = \tau\circ \Ad_{p_i} \circ \pi
\end{equation}
as traces on the subgroup $N$ for all $i \in \{1,\ldots,n\}$, see  Lemma \ref{lem:subtraces-projections}. 
Since the   correspondence $p_i \leftrightarrow \psi_i$ is $G$-equivariant we have that $G_1 = \mathrm{stab}_{G}( p_1)$ and that  the group $G$  permutes the central projections $p_1,\ldots,p_n$ transitively via the adjoint action.  In particular    $\tau(p_i)=\tau(p_1^{t_i})=\tau(p_1)$ so that $\tau(p_i)=\frac{1}{n}$ for all $i \in \{1,\ldots,n\}$.

We   construct a trace representation $(M_1, \pi_1, \tau_1)$ of the finite index subgroup $G_1$ as follows. Take
\begin{equation}
M_{1}=p_{1} M p_{1}, \quad  \pi_{1}=\Ad{}_{p_{1}}\circ\pi \quad \text{and} \quad 
\tau_{1}=\frac{1}{\tau(p_1)} \tau\circ\Ad_{p_{1}}.
\end{equation}
Note that
\begin{equation}
\Ad{}_{p_1}(\pi(G)) = p_1\pi(G)p_1=p_1\pi(G_1)p_1= \pi_1(G_1).
\end{equation}
The von Neumann subalgebra $M_1$ is a factor and $M_1$ is generated by $\pi_1(G_1)$    \cite[Chapter 1 \S2]{dixmier2011neumann}. We conclude that $(M_1, \pi_1, \tau_1)$ is   a genuine trace representation of the subgroup $G_1$.
Consider the character $\varphi_1 = \tau_1 \circ \pi_1$ of the subgroup $G_1$. It extends the trace $\psi_1$ in the sense that
\begin{equation}    \varphi_1(n)=\frac{\tau(p_1\pi(n)p_1)}{\tau(p_1)} =\frac{(\varphi_{|N})_{p_1}(n)}{\tau(p_1)}  =\psi_1(n)
\end{equation}
holds true for all elements $n \in N$.

To conclude the proof   it remains to show that $\varphi = \Ind{G}{G_1}{\varphi_1}$. With this goal in mind consider the von Neumann algebra $\mathcal{M}_{n}(M_{1})$ of all $n$-by-$n$ matrices with entries in $M_1$. This von Neumann algebra admits the  normal faithful trace
\begin{equation}
\widehat{\tau}:x\mapsto\frac{1}{n}\sum_{i=1}^{n}\tau_{1}(x_{ii})\quad  \forall x = (x_{ij}) \in \mathcal{M}_{n}(M_{1}).
\end{equation}
On the other hand, the  map $f:M\to\mathcal{M}_{n}(M_{1})$
given by
\begin{equation}
f(x)_{ij}=p_{1}\pi(t_{i}^{-1})x\pi(t_{j})p_{1} \quad \forall x \in M, \; \forall i,j \in \{1,\ldots,n\}
\end{equation}
 is a normal unital $*$-homomorphism. Since the von Neumann algebra $M$ is a factor it follows that $f$ must be an embedding. Therefore  $\widehat{\tau} \circ f$ is a faithful normal trace on the von Neumann algebra $M$. The uniqueness of traces on finite factors implies that $\tau =\widehat{\tau} \circ f$. 
 Write
\begin{align}
\begin{split}
\label{eq:complicated}
    \varphi(g)&=\tau\circ\pi(g)=\hat{\tau}\circ f\circ\pi(g)=\frac{1}{n}\sum_{i=1}^{n}\tau_{1}\left(p_{1}\pi(t_{i}^{-1}gt_{i})p_{1}\right)=\\
    &=\frac{1}{n}\sum_{i=1}^{n}\begin{cases}
    \tau_{1}\circ\pi_{1}(t_{i}^{-1}gt_{i}) & g\in G_{1}^{t_{i}}\\
    0 & g\notin G_{1}^{t_{i}}
    \end{cases}=\Ind {G}{G_{1}}{\varphi_{1}}(g)
\end{split}
\end{align}
for every element $g \in G$.
The first equality on the second line of Equation (\ref{eq:complicated})    is due to the observation that
the condition $g_{0}\notin G_{1}$ implies that   $p_{1}^{g_{0}}=p_{i}$ for some central projection $p_i$ satisfying $p_1 p_i = 0$. In that case  $p_{1}\pi(g_{0})p_{1}=0$. 
\end{proof}

\subsection*{Vanishing of characters} 

The following two lemmas are instrumental towards our study of the character theory of solvable groups.

\begin{lemma}[Bekka]
\label{lemma:Bekka's other lemma}
Let $N$ and $H$ be a pair of normal subgroups of the group $G$ satisfying $N \le H$.
Let $\varphi \in \traces{G}$ be a trace vanishing on $H \setminus N$. Let $g \in G$ be any element.
		If   there is a sequence of elements  $x_n \in G$  such that the commutators  $\left[g,x_n\right] $ belong to pairwise distinct cosets of $N$  and  such that
	$\left[g,x_n\right] \in H$ for all $n \in \NN$
	then $\varphi(g) = 0 $.
\end{lemma}
\begin{proof}
The case where the subgroup $N$ is trivial is  \cite[Lemma 16]{bekka2007operator}. See   \cite[Lemma 4.13]{lavi2020characters} for the general case.
\end{proof}

\begin{lemma}[Bekka--de la Harpe]
\label{lemma:free action on SNAG}
Let $N$ be an abelian normal subgroup of the group  $G$  with Pontryagin dual $\widehat{N}$.   
Let $\varphi \in \traces{G}$ be a trace whose restriction to $N$ corresponds to the Borel probability measure   $\mu_\varphi \in \mathrm{Prob}(\widehat{N})$ via the Fourier transform.  Then $\varphi(g) = 0$ for  every element $g \in G$ that acts $\mu_\varphi$-essentially freely on $\widehat{N}$.
\end{lemma}
\begin{proof}
This statement is   part of the  \enquote{second case} of the proof of \cite[Theorem 12.D.1]{bekka2020unitary}. It  deals specifically with the Baumslag--Solitar group $G = \ZZ \ltimes \ZZ\left[1/n\right]$ and its normal subgroup $N = \ZZ\left[1/n\right]$. However the same argument applies verbatim in the general setting.
\end{proof}

\section{Characters of metabelian groups}

We study the character theory of metabelian groups and show that all of their characters   are induced from   abelian subquotients. 

\begin{theorem}
\label{thm:characters of metabelian groups}
Let $G$ be a metabelian group admitting  an   abelian normal subgroup $N$ such that $G / N$ is abelian. Then any character of the group $G$ is induced from a trace on the abelianization $H^{\mathrm{ab}}$ of some subgroup $H$ satisfying  $N \le H\le G$.
\end{theorem}
\begin{proof}
Let $\varphi $ be any character of the metabelian group $G$. 
We may assume without loss of generality that the character $\varphi$ is faithful, for otherwise we may regard $\varphi$  as a faithful character of the metabelian quotient group $G / \ker \varphi$.

Consider the ergodic $G$-invariant Borel probability measure $\mu_\varphi$ on the Pontryagin dual $\widehat{N}$ associated to the 
 Fourier transform   of  the restriction $\varphi_{|N
}$.
Denote $Q = G/N$ and let $\Sub{Q}$ be the Chabauty space of all subgroups of $Q$.  Consider the Borel stabilizer map
\begin{equation} \mathrm{stab}: \widehat{N} \to \Sub{Q}, \quad \chi \mapsto \mathrm{stab}_Q(\chi) \quad \forall \chi \in \widehat{N}.
\end{equation}
This form of the stabilizer map with range $\Sub{Q}$   is well defined as the normal subgroup $N$ is abelian. Moreover the map $\mathrm{stab}$ is $G$-invariant since the conjugation action of the group $G$ on its quotient $Q$ is trivial. The $G$-ergodicity of the   measure $\mu_\varphi$ implies that there is some normal subgroup $M$ with $N \le M \le G$   such that  
$M = \mathrm{stab}_G(\chi)$ holds true for $\mu_\varphi$-almost every  character $\chi\in \hat{N}$. We obtain  that  $\varphi=\widetilde{\varphi}_{|M} = \Ind{G}{M}{(\varphi_{|M})}$ according to 
  Lemma \ref{lemma:free action on SNAG}. 
  
  Observe that any    pair of elements $m \in M$ and $n \in N$ satisfy $\left[m,n\right] \in N$ and
\begin{equation}
    \varphi([m,n])= \int_{\widehat{N}}\chi([m,n])\, \mathrm{d}\mu_\varphi(\chi)=\int_{\widehat{N}}\chi^m(n)\chi^{-1}(n)\, \mathrm{d}\mu_\varphi(\chi)=1.
\end{equation}
Since the character $\varphi$ is faithful by our assumption it follows that $\left[M,N\right] = \{e\}$.  In other words $M \le \mathrm{C}_G(N)$ and the subgroup $M$ is two-step nilpotent.

Let $\Sub{M}$ denote the Chabauty space of all  subgroups of the group $M$. 
Consider the   Borel map
\begin{equation}
\mathfrak{Z} : \chars{M} \to \Sub{M}, \quad\theta \mapsto \mathrm{Z}(M/\ker \theta).
\end{equation}
Note that $N \le \mathfrak{Z}(\theta)$  holds true for all characters $\theta \in \chars{M}$. In particular  the map $\mathfrak{Z}$ is in fact $G$-invariant. It follows that the map $\mathfrak{Z}$ is  essentially constant with respect to the ergodic $G$-invariant Borel probability measure on the space $\chars{M}$ associated via the Fourier transform to the restriction $\varphi_{|M}$. Therefore there is some subgroup $N \le H \le M$ satisfying  $H=\mathfrak{Z}(\theta) = \mathrm{Z}(M/\ker \theta)$ almost surely.

  Howe's lemma implies that any character $\theta$ of the two-step nilpotent group $M$ is induced from  its central subgroup $H$, see e.g. \cite[Proposition 2.6]{carey1984characters}. We conclude that $\varphi =\widetilde{\varphi}_{|H}= \Ind{G}{H}{(\varphi_{|H})}$ where $\varphi_{|H}$ is  a  $G$-invariant trace  on the subgroup $H$. The restriction   $\varphi_{|H}$ factorizes through the abelian subquotient $H^{\mathrm{ab}}$ as required.
\end{proof}

The above is a restatement of Theorem \ref{thm:intro:characters of metabelian groups} of the introduction. 

\section{Characters of Noetherian groups}

Recall that a group is called \emph{Noetherian} if it satisfies the ascending chain condition on subgroups.  This condition is equivalent to saying that every subgroup is finitely generated. A solvable group is Noetherian  if and only if it is polycyclic. 

In the current section we study the character theory of Noetherian groups in general and of virtually nilpotent groups in particular.

\subsection*{Totally faithful characters}

We introduce a property of characters which is a strengthening of being faithful.  It will play a central role in our analysis. 


\begin{definition*}
A character $\varphi$ of a group $G$ is  \emph{totally faithful} if $\varphi$ is not induced from a non-faithful character of any finite index subgroup.
\end{definition*}

The usefulness of  totally faithful characters  has to do with the fact that any   character of a
Noetherian group is induced from a totally faithful character of some subquotient.

  \begin{prop}[Noetherian induction principle]
\label{prop:induced from totally faithful}
Let $G$ be a Noetherian group and  $\varphi \in \chars{G}$  a character. Then there is a finite index subgroup $H $ of $  G$ and a character $\psi \in \chars{H}$ satisfying $\varphi = \Ind{G}{H}{\psi}$ such that $\psi$ is totally faithful when regarded as a character of the subquotient $H / \ker \psi$.
\end{prop}
\begin{proof}
We   construct  inductively  a descending sequence of finite index subgroups $G_i \le G$  equipped  with characters $\varphi_i \in \chars{G_i}$ satisfying $\varphi = \Ind{G}{G_i}\varphi_i$ for all $i \in \NN$.
The fact that the group $G$ is Noetherian will be used  to ensure that this process terminates after finitely many steps and that we arrive at some totally faithful character. 
The base of the induction is given by taking the group $G_0 = G$ and the character $\varphi_0 = \varphi \in \chars{G}$. 

Let us consider the induction step. Assume that $G_i \le G$ is some finite index subgroup and  $\varphi_i \in \chars{G_i}$ is a character satisfying $\varphi = \Ind{G}{G_i}{\varphi_i}$ for some $i\in\NN$. There are two possible cases:
\begin{itemize}
    \item If $\varphi_i$ is totally faithful regarded as a character of the subquotient $G_i / \ker \varphi_i$ then we  may conclude the proof  taking $H = G_i$ and $\psi = \varphi_i$.
    \item Otherwise there exist  a further finite index subgroup $G_{i+1} \le G_i$ containing $\ker \varphi_i$ and a character $\varphi_{i+1} \in \chars{G_{i+1}}$ satisfying $\ker \varphi_i \lneq \ker \varphi_{i+1}$ and $\varphi_i = \Ind{G_i}{G_{i+1}}{\varphi_{i+1}}$. Therefore $\varphi = \Ind{G}{G_{i+1}}{\varphi_{i+1}}$ by induction in stages (Lemma \ref{lem:induction in stages}).
\end{itemize}

Proceed inductively to define the finite index subgroups $G_i$ and the characters $\varphi_i \in \chars{G_i}$ for all $i \in \NN$ as in the above paragraph.   As the group $G$ is Noetherian, the resulting ascending sequence of kernel subgroups $\ker \varphi_i $   must eventually stabilize.  This completes the proof.
\end{proof}

\begin{lemma}
\label{prop:restriction of totally faithful is almost surely faithful}
Let $G$ be a Noetherian group
and $\varphi \in \chars{G}$ be a    totally faithful character. Let    $N$ be a normal subgroup of $G$ such  that $\mu_\varphi \in  \mathrm{Prob}(\Ch{N})$  is  the Fourier transform of the restriction $\varphi_{|N}$. Then $\mu_\varphi$-almost every character of $N$ is faithful.
\end{lemma}

\begin{proof}
Let $\Sub{N}$ denote the Chabauty space of all subgroups of $N$.  As the group $N$ is Noetherian the space $\Sub{N}$ is countable.  The group $G$ acts on the space $\Sub{N}$ by conjugation. There is a natural $G$-equivariant kernel map
\begin{equation}
 \kappa  : \chars{N}   \to \Sub{N}, \quad  \chi \mapsto \ker \chi.
\end{equation}
 The pushforward measure $ \kappa _* \mu_\varphi$ is a $G$-invariant ergodic Borel probability measure on the countable space $\Sub{N}$. As such there are finitely many normal subgroups $K_1,\ldots,K_m$ of $N$ such that $G$-acts transitively on this family and $\mu_\varphi$ is the uniform probability measure on this orbit. In particular $\kappa_* \mu_\varphi(\{K_i\}) = \frac{1}{m}$ for all $i$.
 
 Write   $\varphi_{|N} = \frac{1}{m}\left( \psi_1 + \cdots + \psi_m\right)$ where each trace   $\psi_i \in \traces{N}$ is given by
 \begin{equation} 
 \psi_i = m \cdot \int_{\kappa^{-1}(K_i)} \psi \, \mathrm{d}\mu_\varphi(\psi)
 \end{equation}
 for   $i\in\{1,...,m\}$.  The traces  $\psi_1,...,\psi_m$ are pairwise disjoint   and the group  $G$ permutes them transitively. 
 It follows from Theorem \ref{thm:Mackey intro}   that $\varphi = \Ind{G}{\mathrm{N}_G(K_1)}{\varphi_1}$ for some character $\varphi_1 \in \chars{\mathrm{N}_G(K_1)}$ extending the trace $\psi_1$. 
 
The assumption that the   the character $\varphi$ is  totally faithful implies that   the character   $\varphi_1$ must be faithful. This means that the subgroup $K_1$ is trivial. In other words   $m=1$ so that    $\mu_\varphi$-almost every  character $\psi \in \chars{N}$ is faithful.
\end{proof}

Rather than assuming that the group $G$ itself is Noetherian in Lemma \ref{prop:restriction of totally faithful is almost surely faithful}  it would have sufficed to assume that its normal subgroup $N$ is Noetherian.  

Total faithfulness is useful to obtain   a vanishing criterion for characters of Noetherian groups.

\begin{lemma}\label{lem:vanishing-totally-faithful-centralizer}
Let $G$ be a Noetherian  group and $\varphi \in \chars{G}$ be a   totally faithful character. Let $N$ be an abelian   normal   subgroup of $G$. Then $\varphi(g) = 0 $ for any element $g \in G \setminus \mathrm{C}_G(N)$.
\end{lemma}

\begin{proof}

Let $g\in G$ be any element satisfying $g \notin  \mathrm{C}_G(N)$. We claim that $g$ acts $\mu_\varphi$-essentially freely on the Pontryagin dual $\widehat{N}$ where $\mu_\varphi \in \mathrm{Prob}(\widehat{N})$ is the measure corresponding to the    restriction $\varphi_{|N}$ via the Fourier transform. To establish this claim assume towards contradiction that
 \begin{equation}
     \mu_\varphi(\{\chi \in \widehat{H} \: : \: \chi^g = \chi\}) > 0.
 \end{equation} 
 Since the character $\varphi$ is totally faithful we know that $\mu_\varphi$-almost every multiplicative character $\chi \in \widehat{N}$ is faithful by Lemma \ref{prop:restriction of totally faithful is almost surely faithful}. In particular there exists some faithful multiplicative character $\chi \in \widehat{N}$ with $\chi^g = \chi$. So
 \begin{equation}
    \chi^g = \chi \quad \Rightarrow \quad \chi(n^g) = \chi (n) \quad \forall n \in N \quad \Rightarrow \quad n^g = n \quad \forall n \in N.
 \end{equation}
This means that  $g \in   \mathrm{C}_G(N)$ contrary to the assumption. The claim follows. We conclude that $\varphi(g)=0$ by Lemma \ref{lemma:free action on SNAG} on vanishing of characters.  
\end{proof}

\subsection*{Virtually nilpotent groups}
Let us  review   the character theory of nilpotent groups alongside some new observations. Nilpotent groups are  a basic   special case of polycyclic ones. More importantly, the Fitting subgroup of a general polycyclic group is nilpotent and it plays a crucial role in our approach.

\begin{theorem}[Howe \cite{howe1977representations}]
\label{thm:nilpotent are FC -induced}
Let $G$ be a finitely generated nilpotent group. Then any faithful character $\varphi \in \chars{G}$ is induced from $\FC{G}$. 
\end{theorem}
 
The statement of Theorem \ref{thm:nilpotent are FC -induced} holds true more generally for any   nilpotent group $G$ such that the quotient  $G/\FC{G}$  has \emph{finite rank}, in the sense that any finitely generated subgroup of $G$ is contained in some $r$-generated subgroup for some fixed rank $r\in\NN$  
\cite[Theorem 4.5]{carey1984characters}.

The rest of this section is dedicated to the proof of Theorem \ref{thm:characters of virtually nilpotent} which  generalizes  Theorem \ref{thm:nilpotent are FC -induced} to the virtually nilpotent case. We begin with a few useful lemmas.

\begin{lemma}[Kaniuth]
\label{lemma:Kaniuth on induced}
Let $G$ be a finitely generated group admitting a nilpotent normal subgroup $H$ of finite index. Let $\varphi$ be a faithful character of $G$ such that the restriction $\varphi_{|H}$ is   a character of $H$. Then the character $\varphi$ is induced from $\mathrm{FC}(G)$.
\end{lemma}
We emphasize that saying that the restriction $\varphi_{|H}$ is a character means that is not a proper convex combination of traces of $H$.
\begin{proof}[Proof of Lemma \ref{lemma:Kaniuth on induced}]
This is \cite[Lemma 2]{kaniuth1980ideals}.
\end{proof}

\begin{lemma}
\label{lem:new world order lemma}
Let $G$ be a group admitting a finitely generated virtually nilpotent normal subgroup $N$. Let $\varphi \in \chars{G}$ be a totally faithful character. Then the restriction $\varphi_{|N}$ is induced from $\FC{N}$.
\end{lemma}
\begin{proof}
Let $\mu_\varphi \in \mathrm{Prob}(\chars{N})$ denote the measure corresponding to the restriction $\varphi_{|N}$ via the Fourier transform. The measure  $\mu_\varphi$ is $G$-invariant and ergodic.

Fix a finite index nilpotent subgroup $H$ of $N$ which is normal in $G$.  For every character   $\zeta \in \chars{N}$  the restriction $\zeta_{|H}$ is a convex combination of finitely many characters of $H$.  Choose a  $\mu_\varphi$-measurable mapping taking a character $\zeta \in \chars{N}$ to an arbitrary character $\psi_\zeta \in \chars{H}$ appearing in this  restriction. 

Consider the subgroups
\begin{equation}
S(\zeta) = \mathrm{stab}_N(\psi_\zeta)
\end{equation}
defined for $\mu_\varphi$-almost every character $\zeta$ of $N$.  The fact that the group $N$ has countably many subgroups  combined with the ergodicity of the measure $\mu_\varphi$ implies that the subgroups $S(\zeta)$ all belong to a single conjugacy class.
 This allows us to assume that   $S(\zeta) = S$ for some fixed subgroup $S$ satisfying $H \le S\le N$ up to modifying the $\mu_\varphi$-measurable mapping $\psi_\zeta$ where necessary.

Consider   the collection of characters $M(\zeta)$ given for $\mu_\varphi$-almost every character $\zeta \in \chars{N}$  by
\begin{equation}
\label{eq:M(zeta)}
M(\zeta) = \{\chi \in \chars{S} \: : \: \chi_{|H} = \psi_\zeta, \; \Ind{N}{S}{\chi} = \zeta \}.
\end{equation}
The set $M(\zeta)$ is $\mu_\varphi$-almost surely non-empty by Theorem \ref{thm:Mackey intro} and   finite by \cite{itibader}.
Let  $\chi_\zeta \in M(\zeta)$ be an arbitrary character chosen in a    $\mu_\varphi$-measurable manner.

As   the character  $\varphi$ is totally faithful it follows   that the  character $\psi_\zeta \in \chars{H}$ is  faithful $\mu_\varphi$-almost surely, see Lemma \ref{prop:restriction of totally faithful is almost surely faithful}.  In particular the kernel of the character $\chi_\zeta$  is   finite $\mu_\varphi$-almost  surely so that
\begin{equation}
\FC{S /\ker\chi_\zeta} = \FC{S  }.
\end{equation}
Every   character $\chi_\zeta \in \chars{S}$  restricts $\mu_\varphi$-almost surely to a character of $H$ and is therefore   induced from the subquotient $\FC{S /\ker\chi_\zeta}$  by 
  Lemma \ref{lemma:Kaniuth on induced}. In other words the   character $\chi_\zeta$   is   $\mu_\varphi$-almost surely induced from the subgroup $\FC{S }$.

 Observe that $\FC{S } \le \FC{N}$ as $\left[N:S\right] < \infty$. Therefore $\mu_\varphi$-almost every character $\zeta$ is induced from the subgroup $\FC{N}$ according to the definition of the family $M(\zeta)$ given in Equation (\ref{eq:M(zeta)}) and using induction in stages (Lemma \ref{lem:induction in stages}). It follows that the restriction $\varphi_{|N} = \int \zeta \, \mathrm{d}\mu_\varphi(\zeta)$ vanishes outside the normal subgroup $\FC{N}$.
\end{proof}

The special case of Lemma \ref{lem:new world order lemma} where the group $G$ is  itself virtually nilpotent gives the following.

\begin{cor}
\label{cor:totally faithful is induced from FC}
Let $G$ be a finitely generated virtually nilpotent group. Any totally faithful character $\varphi \in \chars{G}$ is induced from  $\mathrm{FC}(G)$.
\end{cor}

We are ready to obtain the following result.

\begin{theorem}
\label{thm:characters of virtually nilpotent}
Let $G$ be a finitely generated virtually nilpotent group. Then any character $\varphi$ of the group $G$ is induced   from $\FC{H/\ker \varphi}$ for some finite index subgroup $H$ of $G$.
\end{theorem}
\begin{proof} 
Let $G$ be a finitely generated virtually nilpotent group. Consider any character $\varphi$   of $G$. The Noetherian induction principle allows us to find a finite index subgroup $H$ of $G$ and a character $\psi $ of $H$ satisfying $\varphi = \Ind{G}{H}{\psi}$ so that $\psi$ is totally faithful regarded as a character of the subquotient $H/\ker\psi$. The totally faithful character $\psi$ is induced from $\FC{H / \ker \psi}$ according to Corollary \ref{cor:totally faithful is induced from FC}. The statement follows by induction in stages (Lemma \ref{lem:induction in stages}).
\end{proof}

\section{Characters of virtually polycyclic groups}
\label{sec:virtually polycycic}

The goal of the current section is to prove Theorem \ref{Thm:characters of virtually polycyclic groups} from the introduction dealing with characters of virtually polycyclic groups.  

To set the stage for the study of the character theory of  polycyclic groups we first make a brief algebraic digression.

\subsection*{The Fitting subgroup}
\label{sec:Fitting}

Recall that the \emph{Fitting subgroup} $\Fit{G}$ of the group $G$ is the characteristic subgroup generated by all normal nilpotent subgroups of $G$. If the group $G$ is Noetherian  then $\Fit{G}$ is nilpotent \cite[1.2.9]{lennox2004theory}. In particular $\Fit{G}$ is nilpotent provided that  $G$ is virtually polycyclic.

\begin{lemma}
\label{lem:infinite order conjugates prep}
Let $G$ be a virtually polycyclic group. Let $g \in  G$ be an element whose projection to the quotient $G/\Fit{G}$  has infinite order.  Then the subgroup $L = \left<g\right> \cdot \Fit{G}$ is not nilpotent.
\end{lemma}
\begin{proof}
The quotient group $G / \Fit{G}$ is virtually abelian by a theorem of Mal'cev \cite{malcev1951}. Let $H $ be a normal finite index subgroup of $G$ containing $\Fit{G}$ such that the quotient group $H/\Fit{G}$ is abelian. The nilpotent subgroup $\Fit{H}$ is  normal in $G$ so that $\Fit{H} \le \Fit{G}$. On the other hand $L \cap H \lhd H$ and $\Fit{H} \lneq L \cap H$ by the assumption on the element $g $. We conclude that the subgroup $L \cap H$ is not nilpotent. Therefore the group $L$ itself is not  nilpotent.
\end{proof}

We are interested in  studying the center  of the Fitting subgroup $\Fit{G}$. Note that if $G$ is solvable then  $\mathrm{Z}(\Fit{G}) = \mathrm{C}_G(\Fit{G}) $ \cite[1.2.10]{lennox2004theory}. The following result gives more refined information.

\begin{lemma}
\label{lem:infinite order conjugates inside Fitting}
Let $G$ be a virtually polycyclic group. Let $g \in  \mathrm{C}_G(\mathrm{Z}(\Fit{G}))$ be an element whose projection to the quotient $G/\Fit{G}$  has infinite order. 
Then there is an element $x \in \Fit{G}$ such that the commutators $\left[g,x^n\right]$  are pairwise distinct modulo the  subgroup $\mathrm{FC}(\Fit{G})$ for all $n \in \NN$.
\end{lemma}
\begin{proof}
Consider the upper central series
\begin{equation}
\label{eq:central series}
\{e\}  = Z_0 \le Z_1 = \mathrm{Z}(\Fit{G}) \le Z_2 \le \cdots \le Z_k = \Fit{G}
\end{equation}
for the nilpotent subgroup $\Fit{G}$ of nilpotence degree $k\in \NN$. An automorphism $\alpha$ of the group $\Fit{G}$ is said to \emph{stabilize} this central series if $\alpha(xZ_i) = xZ_i$ for all $i \in \{0,\ldots,k-1\}$ and all elements $x \in Z_{i+1}$. Any subgroup of $\mathrm{Aut}(\Fit{G})$ which stabilizes a central series for the group $\Fit{G}$ must be nilpotent \cite[1.2.7]{lennox2004theory}.

We know that the subgroup $L = \left<g\right>  \Fit{G}$ is not nilpotent by Lemma \ref{lem:infinite order conjugates prep}. In particular the central quotient  $\overline{L} = L /\mathrm{Z}(L)$ is not nilpotent as well.
Moreover
\begin{equation}
\mathrm{C}_L(\Fit{G}) =     \mathrm{C}_G(\Fit{G}) \cap L = \mathrm{Z}(\Fit{G}) \cap L = \mathrm{Z}(\Fit{G}).
\end{equation}
This means that the central quotient  $\overline{L}$ embeds into $\mathrm{Aut}(\Fit{G})$ via its action by inner  automorphisms. Therefore the element $g$ does not centralize some   factor of the upper central series in Equation (\ref{eq:central series}) above.
We wish to make this observation a bit more precise, as follows.

Write $Z_i / Z_{i-1} = T_i \oplus A_i  $ where $T_i$ is the torsion subgroup of the abelian factor $Z_i / Z_{i-1}$ and $A_i$ is some direct complement for each index $i \in \{1,\ldots,k\}$. 
There is some $m \in \NN$ such that the power $g^m$ centralizes the torsion subgroup $T_i$ and that the condition $\left[g,A_i\right] \subset T_i$ implies $\left[g^{m},A_i\right] \subset Z_{i-1}$ for all $i$. Arguing as in the previous paragraph with respect to the power $g^{m}$, we deduce that the element $g$ satisfies $\left[g,A_i\right] \not\subset T_i$   for some fixed index   $i \in \{2,\ldots,k\}$. 
The possibility of $i=1$ here is excluded by the   assumptions.

Let $z \mapsto \overline{z}$ denote the quotient map from the subgroup $Z_i$ to the abelian factor $Z_i/Z_{i-1}$. Take an element   $x \in Z_i$ such that $\overline{x} \in A_i$ and that the commutator $y = \left[g,x\right]$ has $\overline{y} \notin T_i$.

To conclude the proof it remains to show that the commutators  $y_n = \left[g,x^n\right]$ are pairwise distinct modulo the subgroup $\FC{\Fit{G}}$ for all $n \in \NN$. Note that the center $\mathrm{Z}(\Fit{G})$ has finite index in the subgroup  $\FC{\Fit{G}}$  \cite[p. 98]{kaniuth1980ideals}. In particular the projection of $\FC{\Fit{G}} \cap Z_{i}$ to the factor group $Z_i/Z_{i-1}$  lies inside the torsion subgroup $T_i$. Therefore it will suffice to verify that the $\overline{y}_n$'s are pairwise distinct modulo the torsion subgroup $T_i$ regarded as elements of the subquotient $Z_i/Z_{i-1}$.  Relying on standard commutator identities we   write
\begin{equation} \label{eq:lemma-FFit-Comm}
 y_n = \left[g,x^{n}\right] = \left[g,x^{n-1}\right] \left[g,x\right]^{x^{n-1}} = y_{n-1} y \left[y,x^{n-1}\right]
\end{equation} for all $n \in \NN$. As $[y,x^{n-1}]\in Z_{i-1}$ we have  $\overline{y}_n=\overline{y}_{n-1}\cdot \overline{y}$   by Equation (\ref{eq:lemma-FFit-Comm}). This gives     $\overline{y}_n = \overline{y}^n$  by induction. The proof is complete.
\end{proof}

\subsection*{Crystallographic groups}
\label{page:FFit}
For an arbitrary group $G$ let $\mathrm{F}(G)$ denote the characteristic subgroup\footnote{This subgroup is called the \emph{polyfinite radical} of $G$ and denoted $W(G)$   in \cite{cornulier2015commability}.}   generated by all  finite normal subgroups. If the group $G$ is Noetherian then its subgroup $\mathrm{F}(G)$ is finite. 
Let  $\vFit{G}$ denote the characteristic subgroup of $G$ corresponding to $\mathrm{F}(G/\Fit{G}))$. See \cite{dekimpe1994structure}.

From now on   assume that the group $G$ is virtually polycyclic. This means that $\left[\vFit{G}:\Fit{G}\right] < \infty$ so that $\vFit{G}$ is virtually nilpotent. In fact $\vFit{G}$ is the maximal  virtually nilpotent normal subgroup of $G$.


Recall that the quotient $G/\Fit{G}$ is virtually abelian \cite{malcev1951}. Therefore the quotient  $\Gamma(G) = G/\vFit{G}$ is a virtually abelian group without   non-trivial finite normal subgroups. It follows that $\Gamma(G)$ is  a \emph{crystallographic group} \cite[Theorem 1.1]{dekimpe1994structure}. In other words,   the quotient  $\Gamma(G)$  is a uniform lattice in the group of isometries of the   Euclidean space $\mathbb{E}^d$ in some dimension $d \in \NN$. 
We shall use the following elementary fact regarding crystallographic groups.

\begin{lemma}
\label{lem:infinite order commutators in crystal}
Let $\Gamma$ be a crystallographic group and  $g \in \Gamma$ be a non-trivial torsion element. Then there is a translation    $h \in \Gamma$ so that the commutators $\left[g,h^n\right]$    are pairwise distinct translations in $\Gamma$ for all $n \in \NN$.
\end{lemma}

\begin{proof}
The fixed point set $\mathrm{Fix}(g) = \{ x \in \mathbb{E}^d \, : \, gx = x\}$ of the torsion element $g$ is a proper affine subspace of the Euclidean space $\mathbb{E}^d$. The subgroup of translations in $\Gamma$ acts co-compactly on $\mathbb{E}^d$. Therefore there is some non-trivial translation $h \in \Gamma$ which does not preserve the set $\mathrm{Fix}(g)$. In particular the elements $g$ and $h$ do not commute. The element $t = \left[g,h\right]$ is a non-trivial translation (since  translations form a normal subgroup of $\Gamma$). Using standard commutator identities and relying on the fact that the subgroup of translations is abelian we get
\begin{equation}
    \left[g,h^n\right] = \left[g,h^{n-1}\right] \left[g,h\right]^{h^{n-1}} = \left[g,h^{n-1}\right] t ^{h^{n-1}} = \left[g,h^{n-1}\right] t
\end{equation} 
for all $n \in \NN$.   Arguing by induction gives $\left[g,h^n\right] = t^n  $ for all $n \in \NN$. Therefore the elements $\left[g,h^n\right]$ are   pairwise distinct.
\end{proof}

\subsection*{Characters of virtually polycyclic groups}


We begin by establishing  Theorem \ref{Thm:characters of virtually polycyclic groups} under the additional assumption that the character in question is totally faithful.

\begin{prop}
\label{prop: a totally faithful character is induced from FC}
Let $G$ be a virtually polycyclic group. Then any totally faithful character of $G$ is induced from $\FC{\vFit{G}}$.
\end{prop}
\begin{proof}
Let $\varphi \in \chars{G}$ be a totally faithful character. Fix an arbitrary element   $g\in G\setminus \FC{\vFit{G}}$. We will show that $\varphi(g)=0$. 

To begin with, we know by Lemma   \ref{lem:new world order lemma} that the restriction $\varphi_{|\vFit{G}} $ is induced from the subgroup $\FC{\vFit{G}}$. Therefore we may assume   that  $g\notin \vFit{G}$. Likewise Lemma \ref{lem:vanishing-totally-faithful-centralizer} allows us to assume   that  $g\in  \mathrm{C}_G(Z(\Fit{G}))$.


Assume   that the image of the element $g$ in the crystallographic group quotient $\Gamma = G/\vFit{G}$ has infinite order.
 According to Lemma \ref{lem:infinite order conjugates inside Fitting} there is an element $x \in \Fit{G}$ so that the commutators $\left[g,x^n\right] \in \Fit{G}$ are pairwise distinct modulo the subgroup $\FC{\Fit{G}}$.  However another application of 
 Lemma   \ref{lem:new world order lemma}   with respect to the restriction $\varphi_{|\Fit{G}}$ shows that this restriction  is induced from $\FC{\Fit{G}}$.  We  conclude that $\varphi(g)= 0 $  from the vanishing result for characters stated in Lemma \ref{lemma:Bekka's other lemma}.  
 
 
Finally assume that the image of the element $g$ in the crystallographic group quotient $\Gamma =  G / \vFit{G}$ has finite order.  There exists   an element   $h \in G \setminus \vFit{G}$  projecting to a translation in the crystallographic group quotient $\Gamma$ so that the commutators $\left[g,h^n\right]$ project to pairwise distinct translations in $\Gamma$,  see  Lemma \ref{lem:infinite order commutators in crystal}.  We have seen in the previous paragraph  that the character $\varphi$ vanishes   on all elements of $G$ that project onto a translation element in the crystallographic quotient $\Gamma$. We conclude that  $\varphi(g) =0  $   by relying again  on Lemma \ref{lemma:Bekka's other lemma} (with the normal subgroup of all translations in $\Gamma$ playing the role of the subgroup $H$).
 \end{proof}

We  conclude  the proof of our main result in the character theory of polycyclic groups.

\begin{proof}[Proof of Theorem \ref{Thm:characters of virtually polycyclic groups}]
Let $\varphi \in \chars{G}$ be any character of the  virtually polycyclic group $G$.
According to the Noetherian induction principle (Proposition \ref{prop:induced from totally faithful}) there is some finite index subgroup $H \le G$ and some character $\psi \in \chars{H}$ satisfying   $\varphi = \Ind{G}{H}{\psi}$ such that $\psi$ is totally faithful when regarded as a character of the subquotient $H/\ker \psi$. The character $\psi$ is induced from $\FC{\vFit{H/\ker \psi}}$ according to Proposition \ref{prop: a totally faithful character is induced from FC}. The  result follows by induction in stages (Lemma \ref{lem:induction in stages}).
\end{proof}




\section{Characters of virtually central groups}
\label{sec:fc induced}

The point of departure for the current section is the following well-known fact.

\begin{lemma}[Schur's lemma for characters]
\label{obs:restriction is multiplicative}
The restriction of any character of a countable group $G$ to its center $Z(G)$   is a multiplicative character.
\end{lemma}
\begin{proof}
Let $G$ be a countable group. Consider any character $\varphi \in \chars{G}$ with a corresponding trace representation   $(M,\pi,\tau)$.   The unitary operators $\pi(\mathrm{Z}(G))$ are contained in the center $\mathrm{Z}(M)$ of the von Neumann algebra $M$. Since $\varphi$ is a character the von Neumann algebra $M$ is a factor. Therefore  $\pi(z)$ is a unimodular complex scalar denoted $\chi(z)$ for each element $z\in \mathrm{Z}(G)$. It follows that   $\varphi_{|Z} = \chi$ is   a multiplicative character of the center $\mathrm{Z}(G)$.
\end{proof}

\subsection*{Restriction is a covering map}

Let $G$ be a virtually central group, i.e. a group satisfying $\left[G:\mathrm{Z}(G)\right] < \infty$. 
Fix a central subgroup 
 $Z \le \mathrm{Z}(G)$  with $\left[G:Z\right] < \infty$. 
 We obtain the  restriction map 
\begin{equation}
r: \chars{G} \to \chars{Z}, 
\quad
r : \varphi \mapsto \varphi_{|Z}.
\end{equation}
The restriction map $r$ is continuous, surjective \cite[Lemma 16]{thoma1964unitare} and has finite fibers \cite[Proposition 4.8.(1)]{lavi2020characters}. A deeper property of the map $r $ is as follows.

\begin{prop}
\label{prop:restriction is open}
The restriction map $r : \chars{G} \to \chars{Z}$ is open.
\end{prop}
\begin{proof}
Denote   $Q \cong G/Z$ so that the quotient group $Q$ is finite. Consider the short exact sequence
\begin{equation}
\label{eq:central extension}
1  \to Z \to G \xrightarrow{p} Q \to 1.
\end{equation}
Let $s : Q \to G$ be an arbitrary section to the projection map $p$ satisfying $s(e_Q) = e_G$. Consider the 
cohomology class $\left[c\right] \in \mathrm{H}^2(Q,Z)$ naturally associated   to the central extension given in Equation (\ref{eq:central extension}). The class $\left[c\right]$ is represented by the $2$-cocycle $c : Q \times Q\to Z$ determined by 
\begin{equation} s(q_1)s(q_2) = c(q_1,q_2) s(q_1q_2) \quad \forall q_1,q_2 \in Q.
\end{equation}

Every multiplicative character $\chi \in \chars{Z}$ determines the cohomology class $\left[c_\chi\right] \in \mathrm{H}^2(Q,S^1)$ represented by the $2$-cocycle $c_\chi = \chi \circ c$. We observe that the  second cohomology group 
$ \mathrm{H}^2(Q,S^1)  $ is  finite. Indeed $ \mathrm{H}^2(Q,S^1)  $ can be regarded as a subgroup the   cohomology group $ \mathrm{H}^2(Q,\CC^*)$ and this latter group is finite by
\cite[Chapter 2, Theorem 3.22.(ii)]{karpilovsky1985projective}.
It follows that   the group of coboundaries $\mathrm{B}^2(Q,S^1)$ is an open subgroup of the group of cocycles $\mathrm{Z}^2(Q,S^1)$ when both are regarded as abelian topological groups. 

Note that the correspondence $\eta \mapsto \beta_\eta$ taking a point $\eta \in (S^1)^Q$ to the co-boundary $\beta_\eta \in \mathrm{B}^2(Q,S^1)$ determined by
\begin{equation}
\beta_\eta(q_1,q_2) = \eta(q_1) \eta(q_2) \eta(q_1 q_2)^{-1}
\end{equation}
is an homomorphism of compact abelian groups. This establishes the following   topological group isomorphism 
\begin{equation}
\label{eq:quotient top}
\mathrm{B}^2(Q,S^1) \cong (S^1)^Q/\mathrm{Hom}(Q,S^1).
\end{equation}

Fix a multiplicative character $\chi \in \chars{Z}$  of the central subgroup $Z$ and let $\varphi \in r^{-1}(\chi)$ be any character of the group $G$ that extends $\chi$. Consider  a character $\chi_1 \in \chars{Z}$ which is close to $\chi$, i.e. so that $\chi^{-1} \chi_1$ lies in some small neighborhood of the identity in $\chars{Z}$. Our goal is to find a character $\varphi_1 \in r^{-1}(\chi_1)$ which is close to $\varphi$.

Let $(M,\pi,\tau)$ be the trace representation corresponding to the character $\varphi$. We will    define a new representation $\pi_1$ of the group $G$ into the group of unitaries of the von Neumann algebra $M$ so that the desired character  $\varphi_1$ will be given by $\tau\circ \pi_1$.  
Consider the set-theoretical map $\overline{\pi} : Q \to M$ given by   $\overline{\pi} = \pi \circ s$. The map $\overline{\pi}$ satisfies
\begin{equation}
\label{eq:chi}
 \overline{\pi}(q_1) \overline{\pi}(q_2) = c_\chi(q_1, q_2) \overline{\pi}(q_1 q_2) \quad \forall q_1,q_2\in Q.
\end{equation}
Next, consider the map $\overline{\pi}_1 : Q \to M$ given by 
\begin{equation}
\label{eq:mu}
\overline{\pi}_1(q) = \eta(q) \overline{\pi}(q) 
\end{equation}
for some choice of a map $\eta \in (S^1)^Q$. Define
\begin{equation}
\label{eq:the map pi}
\pi_1 : G \to M, \quad \pi_1(s(q)z) = \overline{\pi}(q) \chi_1(z) \quad \forall q \in Q, z \in Z.
\end{equation}

The map $\pi_1$ defined in Equation (\ref{eq:the map pi}) is a group representation if and only if the condition
\begin{equation}
\label{eq:needed for rep}
\overline{\pi}_1(q_1) \overline{\pi}_1(q_2) = c_{\chi_1}(q_1, q_2) \overline{\pi}_1(q_1 q_2) \quad \forall q_1,q_2 \in Q
\end{equation} 
holds where $c_{\chi_1} = \chi_1 \circ c$.
 Substituting Equation (\ref{eq:mu})  into Equation (\ref{eq:needed for rep})  gives
\begin{equation}
 \eta(q_1) \eta(q_2) \overline{\pi}(q_1) \overline{\pi}(q_2) = \eta(q_1 q_2) c_{\chi_1}(q_1, q_2) \overline{\pi}(q_1 q_2).
\end{equation}
Rearranging  and using the definition of the cocycle $c_\chi$ as in Equation (\ref{eq:chi}) we obtain
that the map $\pi_1$ is a group representation if and only if 
\begin{equation}
\left[c_{\chi_1 \chi^{-1}}\right]  = \left[\beta_\eta\right] = 0 \in \mathrm{H}^2(Q,S^1).
\end{equation}
As the $2$-cocycle $\left[c_\chi\right]$ depends continuously on the choice of the character $\chi \in \chars{Z}$ and as the   co-boundaries $\mathrm{B}^2(Q,S^1)$ form an open subgroup of the group of co-chains $\mathrm{Z}^2(Q,S^1)$ we deduce that $\pi_1$ can be made into   a group representation for all characters $\chi'$ sufficiently close to $\chi$ and for \emph{some} suitable choice of $\eta \in (S^1)^Q$.

In order to find   a character $\varphi_1 \in r^{-1}(\chi_1)$ which is \emph{close} to the given character $\varphi$, the point $\eta$ has to be  close to the identity in the compact abelian group $(S^1)^Q$. This can always be achieved taking into account Equation (\ref{eq:quotient top}) and 
as   a  quotient map of topological groups is open. We conclude that the restriction map $r$ is open.
\end{proof}

\subsection*{Characters of FC-groups}
Recall that a finitely generated group is virtually central if and only if it is an FC-group\footnote{A group $G$ is called an \emph{FC-group} if    every   conjugacy class of $G$ is finite. See \cite[\S15.1]{scott2012group} for more information on FC-groups.}.

\begin{theorem}
\label{thm:restriction is covering}
Let $G$ be a finitely generated FC-group and $Z\le G$ be a central subgroup   satisfying $\left[G:Z\right] < \infty$. Then the restriction map $r : \chars{G} \to \chars{Z}$ is a finite-sheeted cover. 
 \end{theorem}

\begin{proof}[Proof of Theorem \ref{thm:restriction is covering}]
The restriction map $r$ is   continuous,   surjective and has finite fibers. The restriction map $r$ is open by Proposition \ref{prop:restriction is open}. Hence $r$ is a local homeomorphism. 
Generally speaking, any  proper local homeomorphism of connected, locally path-connected metric spaces is  a covering map \cite[p. 303]{lee2010introduction}.
\end{proof}

We turn to studying the  dynamics of the dual action on the space of characters of FC-groups.   It turns out that the action of a certain finite index subgroup of automorphisms   on  each connected component of this space is conjugate to toral automorphisms.

In order to state a more precise conclusion  we need some terminology.
An action $\alpha_1$ on  the torus $\TT^k \cong \RR^k/\ZZ^k$ by toral automorphisms is called a  \emph{finite algebraic  factor} of some other action $\alpha_2$ if there is a  homomorphism $f : \TT^k \to \TT^k$ with finite fibers such that $\alpha_1 \circ f = f \circ \alpha_2$. Two such actions   $\alpha_1$ and $\alpha_2$ are called \emph{weakly algebraically isomorphic}  if each one is a finite algebraic  factor of the other. This is an equivalence relation. It is not hard to see  that if $\alpha_1$ is an finite algebraic  factor of $\alpha_2$ then the two actions $\alpha_1$ and $\alpha_2$ are in fact weakly algebraically isomorphic. For all this  see  \cite[\S2.2.1]{katok2011rigidity}.

\begin{prop}
\label{prop:action of dual of normal FC subgroup}
Let $G$ be a finitely generated FC-group.  Then there exists a finite index subgroup of $\Aut{G}$ whose dual action preserves each connected component of the space $\chars{G}$ and its  action on  each connected component is topologically conjugate to a linear action on a torus. Moreover these linear  actions are weakly algebraically isomorphic to the dual action on the torus $\chars{\mathrm{Z}(G)}$.
\end{prop}

We precede the proof of Proposition \ref{prop:action of dual of normal FC subgroup}   with a general result on equivariant covers of tori.

\begin{lemma}
\label{lem:dynamics on finite covers}
Let $T_1$ and $T_2$ be a pair of    tori. Let $\Gamma$ be a group acting on $T_1$ as well as on $T_2$ by homeomorphisms such that there is a  finite-sheeted $\Gamma$-equivariant covering map $r : T_1 \to T_2$. Assume that the $\Gamma$-action   on the torus $T_2$  is by toral automorphisms. Then $\Gamma$ admits a finite-index subgroup $\Gamma_0$  such that the $\Gamma_0$-action on the torus $T_1$ is topologically conjugate to toral automorphisms. 
\end{lemma}

\begin{proof}
Assume that the two tori $T_1$ and $T_2$ are $k$-dimensional for some $k \in \NN$. Let $p_1 : \RR^k \to T_1$  be the universal covering map of the torus $T_1$. The composition $p_2 = r \circ p_1$ is the universal covering map of the torus $T_2$. 
Consider the finite-index subgroup  $\Lambda = r_* \pi_1(T_1)$ of the fundamental group $\pi_1(T_2) \cong \ZZ^k$. Up to topological conjugacy, we may identify the pair of tori $T_1$ and $T_2$ with the two quotient spaces $\RR^k/\Lambda$ and $\RR^k / \ZZ^k$ respectively. In particular the point $0 \in \RR^k$ satisfies $p_1(0) = \left[0\right] \in \RR^k / \Lambda$ and $p_2(0) = \left[0\right] \in \RR^k / \ZZ^k$.

The action of the group $\Gamma$  on the torus $T_2$ by toral automorphisms  lifts to a linear action  on the universal cover $\RR^k$ that fixes the point $0 \in \RR^k$. Such an action is given by some homomorphism $f : \Gamma \to \mathrm{GL}_k(\ZZ)$. The two covering maps $p_1$ and $p_2$ are  $\Gamma$-equivariant with respect to this action. 


Let $\Gamma_0$  be the stabilizer  of the point $\left[0\right] \in \RR^k / \Lambda \cong T_1$ for the action of the group $\Gamma$ on the torus $T_1$.
Since the $\Gamma$-action on torus $T_2$ fixes the point  $\left[0\right]\in \R^k/\Z^k \cong T_2$ and as the covering $r$ is finite-sheeted we have that $\left[\Gamma:\Gamma_0\right] < \infty$.  The restricted $\Gamma_0$-action on the universal cover $\RR^k$ preserves the $\Lambda$-orbit of the point $0\in\RR^k$. T    his orbit can be naturally identified with $\Lambda$ itself regarded as a subgroup of the universal cover $\RR^d$.
Finally, the restriction $f_{|\Gamma_0} : \Gamma_0 \to \mathrm{GL}_k(\ZZ)$ descends to an action of the subgroup $\Gamma_0$ on the torus $T_1$ by toral automorphisms.
\end{proof}

\begin{proof}[Proof of Proposition \ref{prop:action of dual of normal FC subgroup}]
 Let $d \in \NN$ denote the rank of center $\mathrm{Z}(G)$ of the group $G$. The dual space $\chars{\mathrm{Z}(G)}$ can be identified with a disjoint union of finitely many $d$-dimensional tori. The group $G$ satisfies   $\left[G:\mathrm{Z}(G)\right] < \infty$. Therefore the space $\chars{G}$ is a finite-sheeted cover of $\chars{\mathrm{Z}(G)}$ by Theorem \ref{thm:restriction is covering}. As such $\chars{G}$ is also a disjoint union of   $d$-dimensional tori. The group  $\Aut{G}$ acts on the space $\chars{G}$  by homeomorphisms. Let $\Gamma\leq \Aut{G}$ be a finite index subgroup  that preserves each  connected component of the space $\chars{G}$. We   conclude the proof  by relying on Lemma \ref{lem:dynamics on finite covers} individually with respect to each connected component of this cover. The fact that the  action on each connected component  is weakly algebraically isomorphic to the dual  action on the torus $\chars{\mathrm{Z}(G)}$ follows from the discussion on \cite[p. 54]{katok2011rigidity}.
\end{proof}

\subsection*{Characters with finite index kernel}

 We consider the behaviour of characters  of FC-groups whose kernel has finite index.

 \begin{lemma}
 \label{lem:abelian-finite-index-kernel-G-invariant}
 Let $A$ be discrete abelian group  and $\Gamma \le \mathrm{Aut}(A)$ be any subgroup such that $\Gamma \ltimes A$ is finitely generated. Let $\varphi \in \chars{A}$ be a $\Gamma$-invariant character. Then there is a sequence of $\Gamma$-invariant characters $\varphi_n \in \chars{A}$  with $\left[A:\ker \varphi_n\right] < \infty$ satisfying 
 $\varphi = \lim_n \varphi_n$ in the pointwise topology.
 \end{lemma}
 \begin{proof}
Let $A_\Gamma$ denote the $\Gamma$-invariant subgroup of $A$   generated by all  elements of the form  $\gamma_1 a - \gamma_2 a$ where $a \in A$ and $\gamma_1,\gamma_2 \in \Gamma$.  The group of \emph{co-invariants}  $A/A_\Gamma$ is finitely generated as the group $\Gamma\ltimes A$ is assumed to be finitely generated. Note that  $\varphi_{|A_\Gamma} = 1$.
The Pontryagin dual of the quotient $ A/A_\Gamma$ can be identified with the closed subgroup \begin{equation}
    \widehat{A}^\Gamma = \{\chi \in \widehat{A} \: : \: \gamma \chi = \chi \, \forall \gamma \in \Gamma\}
\end{equation} of the Pontryagin dual $\widehat{A}$   consisting of the  $\Gamma$-invariant characters.  The action  of the group $\Gamma$ on the quotient $A/A_\Gamma$ and on its dual $\widehat{A}^\Gamma$ is trivial.

The character $\varphi$ can   be regarded as a character of the quotient group $A/A_\Gamma$. As the quotient $A/A_\Gamma$ is finitely generated   the   character $\varphi$ is a pointwise limit of characters  of  $A/A_\Gamma$ with finite index kernels. We may lift the limiting characters from $A/A_\Gamma$ back to $A$ to obtain a sequence  $\varphi_n$ of $\Gamma$-invariant characters with finite index kernels converging to $\varphi$.
 \end{proof}

\begin{lemma}
\label{lem:finite index kernel for FC}
Let $G$  be a finitely generated FC-group. Let $\Gamma \le \mathrm{Aut}(G)$ be any subgroup of automorphisms. Then there is a finite index subgroup $\Gamma_0 \le \Gamma$ such that any finitely supported $\Gamma_0$-invariant probability measure  on the space $\chars{G}$  is a weak-$*$ limit   of $\Gamma_0$-invariant probability measures supported on  finitely many characters  whose kernels   all have    finite index in $G$.
\end{lemma}
\begin{proof}
Consider the restriction map $r : \chars{G} \to \chars{\mathrm{Z}(G)}$. A given character $\varphi \in \chars{G}$ satisfies $\left[G:\ker \varphi\right] < \infty$ if and only if $\left[\mathrm{Z}(G): \ker r(\varphi)\right] < \infty$. The latter condition holds true if and only if the restriction $r(\varphi)$ is a torsion element of the dual group $\widehat{\mathrm{Z}(G)}$.

Let $\Gamma_0$ denote the intersection of the subgroup   $\Gamma$ with the finite index subgroup of $\Aut{G}$ provided by Proposition \ref{prop:action of dual of normal FC subgroup}. The group $\Gamma_0$ preserves every connected component of the space $\chars{G}$ and acts on it linearly (up to  topological conjugacy). Additionally it follows from Proposition  \ref{prop:action of dual of normal FC subgroup} that the  preimage under the restriction map $r$ of the torsion subgroup of $\chars{\mathrm{Z}(G)}$ can be identified with the torsion subgroup of every connected component of $\chars{G}$ regarded as a compact abelian group. The desired conclusion follows from Lemma \ref{lem:abelian-finite-index-kernel-G-invariant}. 
\end{proof}

 \section{Dense periodic measures}
 \label{sec:sofic systems}
 
Let $G$ be a discrete group acting on a compact  metrizable space  $X$  by homeomorphisms. The pair $(G,X)$ is called a \emph{topological dynamical system}. We restrict our attention to the situation where the acting group $G$ is amenable. 

 \begin{definition*}
 \label{def:sofic action}
 The system $(G,X)$ has   \emph{dense periodic measures} if every $G$-invariant Borel probability measure on the space $X$ is a weak-$*$ limit of $G$-invariant probability measures with finite supports.
 \end{definition*}
 
For the purpose of showing that the system $(G,X)$ has dense periodic measures it suffices to consider ergodic $G$-invariant probability measures on $X$.
 In that case, it is always possible to   take the limiting   $G$-invariant probability measures to be the uniform probability measures supported on  certain finite $G$-orbits  \cite[\S2]{levit2019infinitely} \footnote{The paper \cite{levit2019infinitely} deals with the special case of the topological dynamical system $(G,\Sub{G})$, however the proof of the above mentioned fact is true  more generally.}.
 
\begin{remark}
Throughout this work it will most often be the case that the space $X$ is a compact abelian group and that the group $G$ is acting on $X$ by group automorphisms.
\end{remark}
\subsection*{Hereditary properties} 

We discuss the behaviour of the density of periodic measures   for topological dynamical systems with respect to some basic operations.

\begin{lemma}
\label{lem:DPM-finite-index}
Let $H \leq G$ a finite index normal subgroup. If $(H,X)$ has dense periodic measures then so does $(G,X)$. 
\end{lemma}
\begin{proof}
Let $\mu$ be an ergodic $G$-invariant probability measure on $X$. Consider the ergodic decomposition of the probability space $(X,\mu)$ with the respect to the $H$-action, see e.g. \cite[\S4.2]{einsiedler2013ergodic}. 
The space $Z$ of the $H$-ergodic components admits an ergodic $G/H$-action. Therefore the space  $Z$ is isomorphic to the coset space $G/K$ with the uniform probability measure for some finite index subgroup $K$ satisfying $H \le K \le G$. It follows that  there is an $H$-ergodic probability measure $\nu$ on $X$ so that 
\begin{equation}\label{eq:DPM-finite-index-mu}
    \mu = \frac{1}{\left[G:K\right]} \sum_{g\in G/K} g_* \nu = \frac{1}{\left[G:H\right]} \sum_{g\in G/H} g_* \nu.
\end{equation}

The system   $(H,X)$ has dense periodic measures by assumption. Therefore there is a sequence $\nu_n$ of finitely supported $H$-invariant probability measures on $X$ converging to $\nu$ in the weak-$*$ topology. The sequence of finitely supported $G$-invariant probability measures
\begin{align}\label{eq:DPM-finite-index-mun}
    \mu_n = \frac{1}{\left[G:K\right]} \sum_{g\in G/K} g_* \nu_n
\end{align}
converges to the measure $\mu$ in the weak-$*$ topology.
\end{proof}

The converse direction of Lemma \ref{lem:DPM-finite-index} seems more challenging, see Question \ref{quest:going up finite index} in the introduction.

\begin{prop} \label{prop:DPM-of-factors}
Let $(G,X)$ and $(G,Y)$ be a pair of topological dynamical systems and  $p : X\to Y$  a continuous,  surjective and $G$-equivariant map. Assume that there is a $G$-equivariant Borel map $Y\to\Prob{X}$ taking a point $y \in Y$ to a probability measure $\nu_y \in \Prob{X}$ with $\mathrm{supp}(\nu_y) \subset p^{-1}(y)$. If   $(G,X)$ has dense periodic measures then so does  $(G,Y)$.
\end{prop}

\begin{proof}
Let $\mu$ be any $G$-invariant probability measure on the space $Y$. We wish to show that $\mu$ is a limit of finitely supported $G$-invariant probability measures. To do so, consider the   $G$-invariant probability measure $\nu$ on the space $X$ given by
\begin{equation}
\nu = \int_Y \nu_y \, \mathrm{d}\mu.
\end{equation} 
Note that   $p_*\nu=\mu$.
As the system $(G,X)$ has dense periodic measures there is a sequence  of $G$-invariant finitely supported probability measures $\nu_n$ on $X$ converging to $\nu$ in the weak-$*$ topology. The  $G$-invariant pushforward  measures $p_* \nu_n$   have finite supports and converge to the measure $\mu$ on the space $Y$ in the weak-$*$ topology.
\end{proof}

\begin{cor} \label{cor:DPM-of-factors}
Let $(G,X)$ and $(G,Y)$ be a pair of topological dynamical systems and  $p : X\to Y$  a continuous,  surjective and $G$-equivariant map. 
Assume    either that
\begin{enumerate}
    \item there is a $G$-equivariant Borel section $s:Y\to X$ of the map $p$, or \label{cor:DPM-of-factors item:section}
    \item the system $(G,X)$ is a compact group extension of the system $(G,Y)$, or \label{cor:DPM-of-factors item:compact}
    \item the fibers of the map $p$ are all finite.\label{cor:DPM-of-factors item:finite-fibers}
\end{enumerate}
 If the system $(G,X)$ has dense periodic measures  then so does the system  $(G,Y)$.
\end{cor}

\begin{proof}
Assume that the system   $(G,X)$ has dense periodic measures. We will deduce the density of periodic measures for  the system $(G,Y)$   from  Proposition \ref{prop:DPM-of-factors} by constructing a suitable map $Y \to \mathrm{Prob}(X)$.

In case (\ref{cor:DPM-of-factors item:section}) let   $s:Y\to X$ be a $G$-equivariant Borel section of the map $p$. It is clear that the map $Y\to \Prob{X}$ defined by $ y \mapsto \nu_y = \delta_{s(y)}$ is  Borel and  $G$-equivariant. It satisfies $\mathrm{supp}(\nu_y)=\{s(y)\}\subseteq p^{-1}(y)$ as required.

In case (\ref{cor:DPM-of-factors item:compact})  there is a compact group $K$ admitting a $G$-equivariant continuous action on the space $X$ such that $Y = X/K$ \cite[p. 15]{glasner2003ergodic}.  There is a Borel measurable section $s : Y \to X$ by the Jankov--von Neumann theorem, see e.g. \cite[Theorem 2.10]{glasner2003ergodic}.  Let $\mu$ denote the Haar probability measure on the compact group $K$. Define a Borel map $Y \to \mathrm{Prob}(X)$ taking each point $y \in Y$ to the pushforward $\nu_y$ of the Haar measure $\mu$ via the action map $K \mapsto Ks(y)$. It remains to verify  that the map $y \mapsto \nu_y$ is $G$-equivariant. Let $y_1,y_2 \in Y $ be a pair of points with $gy_1 = y_2$ for some element $g \in G$. It follows that $s(y_2) = k_0 gs(y_1) = gk_0 s(y_1)$ for some element $k_0 \in K$. To conclude observe that
\begin{equation}
    \nu_{y_2} = \mu *  \delta_{s(y_2)}   = \mu *  \delta_{gk_0 s(y_1)} = g \left(\mu *   k_0 \delta_{ s(y_1)}\right) =  g   \nu_{y_1}.
\end{equation}

In case (\ref{cor:DPM-of-factors item:finite-fibers}) assume that the surjective map $p : X \to Y$ has finite fibers. We need to show that the $G$-equivariant map $Y \to \Prob{X}$ assigning to each point $y \in Y$ the uniform measure $\nu_y$ on the finite set $p^{-1}(y)$ is Borel. 
We may write the compact space $Y$ as a disjoint union $Y = \amalg_{n\in\NN} Y_n$ of $G$-invariant Borel sets so that $|p^{-1}(y)| = n$ for all points $y \in Y_n$. Apply   the Kuratowski--Ryll-Nardzewski selection theorem \cite[Theorem 12.13]{kechris2012classical} to find Borel maps $f_{n,i} : Y_n \to X$ for all $n \in \NN$ and all $i \in \{1,\ldots,n\}$ satisfying
\begin{equation}
p^{-1}(y) = \{f_{n,1}(y),f_{n,2}(y),\ldots,f_{n,n}(y)\} \quad \forall n\in\NN,\,\forall  y \in Y_n.
\end{equation}
For any continuous function $F \in C(X)$ and any point $y \in Y$ the integral $\nu_y(F)$ can be evaluated as 
\begin{equation}
\nu_y(F) = \frac{1}{n} \sum_{i=1}^n F(f_{n,i}(y)) \quad \forall n\in \NN, \, \forall y \in Y_n.
\end{equation}
It follows that the map $y \mapsto \nu_y$ is indeed Borel measurable with respect to the weak-$*$ topology on the space $\Prob{X}$. This concludes the proof.
\end{proof}

\subsection*{The specification property} 

Any $\ZZ^d$-action which enjoys the so called \emph{periodic specification property} has dense periodic measures. For an in-depth discussion   of various flavors of this property we refer the reader to \cite{kwietniak2016panorama}.

Bowen introduced the specification property and established it for  Axiom A diffeomorphisms, a class of dynamical systems including hyperbolic torus automorphisms \cite{bowen1971periodic}.   See also Sigmund's treatment in \cite{sigmund1970generic,sigmund1974dynamical}. Lind and Schmidt obtained  necessary and sufficient conditions for a very general family of commuting automorphisms of compact abelian groups to have specification \cite{lind1999homoclinic}.


\subsection*{Ergodic torus automorphisms}

Non-hyperbolic ergodic torus automorphisms do not  enjoy  specification in the strict sense \cite{lind1979ergodic}. Nevertheless there is a useful periodic analogue of \emph{weak specification} for such systems. It was used by Marcus \cite{marcus1980note} to obtain the following.

\begin{theorem}[Marcus]\label{thm:marcus}
Let $A$ be an  automorphism of the   torus $\TT^k$ ergodic\footnote{A torus automorphism is ergodic with respect to the Haar measure  if and only if all of its eigenvalues are not roots of unity.} with respect to the Haar measure.  Then the topological dynamical system $(\ZZ,\TT^k)$ has dense periodic measures.
\end{theorem}

\subsection*{Bernoulli shifts}

Topological dynamical systems arising from Bernoulli shifts have dense periodic measures.  Most of these  facts are well-known and appear in the literature.

\begin{prop}
\label{prop:Bernoulli shift has dense periodic measures}
Let $G$ be any residually finite amenable group. Then for any compact  set $K$ the Bernoulli system $(G, K^G)$ has dense periodic measures with respect to the shift action.
\end{prop}
The case where the acting group is $\ZZ$ was established in \cite[Theorem 3.3]{parthasarathy1961category}. 
 The proof we present below relies on  the pointwise ergodic theorem for amenable groups \cite{lindenstrauss2001pointwise} in the manner  of  \cite[Lemma 1]{sigmund1970generic} and  \cite[Main Theorem]{marcus1980note}. Moreover it uses ideas from  \cite[Theorem 3.10]{levit2019infinitely}.




\begin{proof}[Proof of Proposition \ref{prop:Bernoulli shift has dense periodic measures}]
Fix  a descending   chain  $N_i$ of finite index normal subgroups  of the group $G$ with trivial intersection. The main result of \cite{weiss2001monotileable} allows us to find  a F\o{}lner   sequence $F_i$ of coset transverals to the subgroups $N_i$. For each index $i \in \NN$ and every  element $g \in G$ let   $f_i(g) \in F_i$ denote the unique group element satisfying $f_i(g) N_i = g N_i$.
Given  any point $x \in K^G$ we  define a sequence of new points $x_i \in K^G$ via  
\begin{equation}
    x_i(g) = x( f_i(g)) \quad \forall g \in G.
\end{equation}  
 Observe that the point $x_i$ is $N_i$-periodic for every $i \in \NN$, namely $gx_i = x_i$ holds true for all elements $g \in N_i$.   
 
 Let $\mu$ be an arbitrary $G$-invariant probability measure on the space $K^G$. We may assume without loss of generality that the measure $\mu$ is ergodic. The sequence of finitely supported probability measures $\mu_{x,i} = \frac{1}{|F_i|} \sum_{g \in F_i} \delta_{gx}$ weak-$*$ converges to the probability measure $\mu$ for $\mu$-almost every point $x $ by the pointwise ergodic theorem, see \cite[Theorem 3.9]{levit2019uncountably}for details.
On the other hand,
  as the sequence $F_i$ is  \folner,   the sequence of $G$-invariant finitely supported  probability measures $\nu_{x,i}  = \frac{1}{|F_i|} \sum_{g \in F_i} \delta_{gx_i}$ satisfies  $d(\mu_{x,i},\nu_{x,i}) \to 0$ with respect to any compatible metric $d$ on the space $\mathrm{Prob}(K^G)$ and any point $x \in K^G$.  The desired conclusion follows.
  \end{proof}

\begin{prop}
\label{prop:p mult has dense periodic measures}
Let $\mathrm{S}_p = \varprojlim \RR/p^n \ZZ$ be $p$-adic solenoid. Consider the dynamical system $(\ZZ,\mathrm{S}_p)$ where the group $\ZZ$  is acting via multiplication by $p$. The  system $(\ZZ,\mathrm{S}_p)$ has dense periodic measures.
\end{prop}
\begin{proof}
Let $X_p = \{0,\ldots,p-1\}^\ZZ$ be a Bernoulli system with the shift action.
There is a $\ZZ$-equivariant continuous surjection $X_p \to S_p$  with finite fibers. We conclude that the  system $(\ZZ,\mathrm{S}_p)$ has dense periodic measures  by combining Proposition \ref{prop:Bernoulli shift has dense periodic measures} with Corollary \ref{cor:DPM-of-factors}.
\end{proof}


\section{Characters and Hilbert--Schmidt stability}
\label{sec:HS stability}

We formulate the precise definition of Hilbert--Schmidt stability for discrete groups.  Let $\mathrm{U}(n)$ denote the unitary group of degree $n$ for each $n \in \NN$. The \emph{normalized Hilbert--Schmidt norm}\footnote{By abuse of notation, we omit the   the index $n$ from the notations of the Hilbert--Schmidt norm and metric.}   on the group $\mathrm{U}(n)$ is 
\begin{equation}
\|A\|_\textrm{HS} = \sqrt{\frac{1}{n}\mathrm{tr}(A^* A)} \quad \forall A \in \mathrm{U}(n).
\end{equation}
The corresponding \emph{normalized Hilbert--Schmidt metric}
on the group $\mathrm{U}(n)$ is given by
\begin{equation}
d_ \textrm{HS}(A,B) = \|A-B\|_\textrm{HS} \quad \forall A,B \in \mathrm{U}(n).
\end{equation} 
This metric is bi-invariant in the sense that
\begin{equation}
    d_ \textrm{HS}(A,B) = d_ \textrm{HS}(CAD,CBD)    \quad \forall A,B,C,D \in \mathrm{U}(n).
\end{equation}

\begin{definition*}
Let $G$ be a discrete group. An \emph{asymptotic homomorphism} of the  group $G$ is a sequence of set-theoretic maps $f_n : G \to \mathrm{U}(n)$ for all $n\in \NN$ satisfying
\begin{equation}
d_\textrm{HS}(f_n(g) f_n(h), f_n(gh)) \xrightarrow{n\to\infty} 0 \quad \forall g,h \in G.    
\end{equation}
The group $G$ is  \emph{Hilbert--Schmidt stable} if for any asymptotic homomorphism $f_n : G \to \mathrm{U}(n)$ there is a sequence of   group homomorphisms $\varphi_n : G \to \mathrm{U}(n)$ such that
\begin{equation}
d_\textrm{HS}(f_n(g), \varphi_n(g)) \xrightarrow{n\to\infty} 0 \quad \forall g \in G.
\end{equation}
\end{definition*}

The Hadwin--Schulman criterion \cite{hadwin2018stability} relates Hilbert--Schmidt stability to characters. We proceed with a detailed analysis of this criterion.

\subsection*{Finite dimensional traces}

Let $G$ be a discrete group. A trace representation $(M,\pi,\tau)$ of the group $G$ is called \emph{finite dimensional} if the von Neumann algebra $M$ is finite dimensional. A trace on the group $G$ is called \emph{finite dimensional} if the trace representation corresponding to it via   Theorem \ref{thm:thoma correspondence } is finite dimensional. 
We denote by $\Trfd{G}$ the set of all finite dimensional traces on the group $G$.

\begin{lemma}\label{lem:fd-traces-face}
$\Trfd{G}$ is a face of the convex set $\Tr{G}$.
\end{lemma}
A convex subset $F$ of a convex set $C$ is called  a \emph{face} if $y,z \in F$ whenever $x = ty + (1-t) z$ for some points $x \in F$,  $y,z \in C$ and some $ 0 < t < 1$, see \cite[Definition 16.5]{phelps2001lectures}.
\begin{proof}[Proof of Lemma \ref{lem:fd-traces-face}]
Consider a pair of traces  $\varphi_1,\varphi_2 \in \traces{G}$.  Let $\varphi$ be a non-trivial convex combination of the traces $\varphi_1$ and $\varphi_2$ so that  
$\varphi=t\varphi_1+(1-t)\varphi_2$  for some $t\in(0,1)$. 

We may write $\varphi_1 = s_1  \varphi'_1 + (1-s_1) \psi  $ and $\varphi_2 = s_2 \varphi'_2  + (1-s_2) \psi$ for some real numbers $s_1,s_2 \in \left[0,1\right]$ such that the three traces $\varphi'_1, \varphi'_2$ and $\psi$ are pairwise disjoint. Let   $(M_1,\pi_1,\tau_1),(M_2,\pi_2,\tau_2)$ and  $(M_3, \pi_3, \tau_3)$ be the trace representations corresponding to $\varphi'_1, \varphi'_2$ and $\psi$ respectively. A direct computation shows that the trace representation corresponding to the trace $\varphi$ is given by $(M,\pi,\tau)$  where 
\begin{equation}
    \pi=\pi_1\oplus \pi_2 \oplus \pi_3 \quad \text{and} \quad \tau=ts_1\tau_1\oplus(1-t)s_2\tau_2 \oplus (t(1-s_1)+(1-t)(1-s_2)) \tau_3.
\end{equation}
 The von Neumann algebra $M$ is given by  $M=\pi(G)''\subseteq M_1\oplus M_2 \oplus M_3$. The fact that the  traces $\varphi'_1, \varphi'_2$ and $\psi$ are pairwise disjoint implies that $M = M_1 \oplus M_2 \oplus M_3$. A similar argument shows that the von Neumann algebras corresponding to the two traces $\varphi_1$ and $\varphi_2$ are $M_1 \oplus M_3$ and $M_2 \oplus M_3$ respectively.

 We conclude that the trace $\varphi$ is finite dimensional if and only if both traces $\varphi_1$ and $ \varphi_2$ are finite dimensional. This means that $\Trfd{G}$ is convex and is a  face of the simplex $\traces{G}$, as required.
\end{proof}

The Hadwin--Shulman criterion is given in terms of the functions   $\frac{1}{n}\tr\circ \pi$ where $\pi : G \to U(n)$ is some finite dimensional unitary representation. Any such function is a finite dimensional trace on $G$. Indeed its   corresponding trace representation is given by $(\pi(G)'',\pi,\frac{1}{n}\tr)$. However, not every finite dimensional trace is a normalized trace of some finite dimensional unitary representation. An example   for   the group $G = \ZZ / 2\ZZ$ is provided by the trace $\varphi = t \chi_1 + (1-t) \chi_{-1}$ for any irrational  $t\in [0,1]$. This nuance is clarified by the following lemma.

\begin{lemma}\label{lem:finite-dimensional-traces-rational-coefficients}
Let $\varphi : G \to \CC$ be a function. The following  two conditions are equivalent.
\begin{enumerate}
    \item $\varphi = \frac{1}{\dim\pi}\tr\circ \pi$ for some finite dimensional unitary representation $\pi$.
    \item $\varphi $ is a finite rational convex combination   of finite dimensional characters. 
\end{enumerate}  
\end{lemma}

\begin{proof}
Assume that the function $\varphi$ is given by $ \frac{1}{\dim\pi}\tr\circ \pi$ for some finite dimensional unitary representation $\pi$ of the group $G$. Let $\pi=\bigoplus_{i=1}^d \pi_i$ be the decomposition of the representation $\pi$ into isotypic components, that is, each subrepresentation $\pi_i$ is the sum of all the irreducible subrepresentations of $\pi$ sharing  the same isomorphism type.
By Schur's Lemma each von Neumann algebra $\pi_i(G)''$ is a factor. Therefore each function $\varphi_i=\frac{1}{\dim\pi_i}\tr\circ \pi_i$ is a finite dimensional character. We have
\begin{equation}    \varphi=\frac{1}{\dim\pi}\tr\circ \pi= \frac{1}{\dim\pi} \sum_{i=1}^d \tr\circ \pi_i = \sum_{i=1}^d \frac{\dim \pi_i}{\dim\pi} \varphi_i 
\end{equation}
so that the function $\varphi$ is   a finite rational convex combination of finite dimensional characters.

Consider the   converse direction. Assume that the function $\varphi$ is a finite rational convex combination of finite dimensional characters. 
Therefore $\varphi=\frac{1}{d}\sum_{i=1}^d\varphi_i$ for some   $d\in \NN$ and some   finite dimensional characters $\varphi_1,...,\varphi_d \in \Ch{G}$, possibly with repetitions. Let $(M_i,\pi_i,\tau_i)$ be a trace representation corresponding to each character $\varphi_i$. Since $\varphi_i$ is a finite dimensional character the von Neumann algebra $M_i$ is a finite dimensional factor. In other words $M_i \cong \mathrm{M}_{\dim \pi_i}(\CC)$.  Therefore we must have   $\tau_i = \frac{1}{\dim\pi_i}\tr$ and $\varphi_i=\tau_i \circ \pi_i$.  Let $m\in \NN$ be the least  common multiple of the dimensions $\dim\pi_1,...,\dim\pi_d$. Consider the finite dimensional unitary representation $\pi=\bigoplus_{i=1}^d  \frac{m}{\dim\pi_i} \pi_i$. Its dimension is $dm$. We conclude that 
\begin{equation}
    \varphi=\frac{1}{d}\sum_{i=1}^d \varphi_i=\frac{1}{\dim \pi}\sum_{i=1}^d \frac{m}{\dim\pi_i} \ \tr\circ \pi_i=   \frac{1}{\dim \pi}\tr\circ \pi
\end{equation}
as required.
\end{proof}

The above analysis can be used to reformulate the Hadwin--Shulman criterion in the following manner.

\begin{cor}[Hadwin--Shulman \cite{hadwin2018stability}]
\label{cor:Hadwin Shulman}
Let $G$ be an amenable group. The following conditions are equivalent.
\begin{enumerate}
    \item The group $G$ is Hilbert--Schmidt stable.
    \item Any character of the group $G$ is a pointwise limit of normalized traces of finite dimensional unitary representations.
    \item Any trace on the group $G$ is a pointwise limit of normalized traces of finite dimensional unitary representations.
    \item The face $\Trfd{G}$ is dense in  the compact convex set   $\Tr{G}$. 
\end{enumerate} 
\end{cor}

Indeed the equivalence of  (2), (3) and (4) holds true in general and does not depend on amenability.

\begin{proof}[Proof of Corollary \ref{cor:Hadwin Shulman}]
The equivalence of conditions $(1)$, $(2)$ and $(3)$  is established in  \cite[Theorem 4]{hadwin2018stability} combined with  \cite[Lemma 1]{hadwin2018stability}. The implication $(3)\Rightarrow (4)$ is clear   as normalized traces of finite dimensional unitary representations are  a special form of finite dimensional traces.  
It remains to establish the implication  $(4)\Rightarrow (3)$. Recall that the subset $\Trfd{G}$ of the finite dimensional traces is   a face of the convex set $\traces{G}$ of all traces, see Lemma \ref{lem:fd-traces-face}. In particular the extreme points of $\Trfd{G}$  are precisely the finite dimensional \emph{characters}. The Krein--Milman theorem implies that the subset of the face $\Trfd{G}$ consisting of all finite rational   convex combinations of finite dimensional characters is dense in $\Trfd{G}$. But this subset coincides with the subset of all normalized traces of finite dimensional unitary representations, see Lemma \ref{lem:finite-dimensional-traces-rational-coefficients}. This completes the proof. 
\end{proof}

\subsection*{Dense periodic measures   on character spaces}

\begin{lemma}\label{lem:f.d-trace-finitely-supported-measure}
Let $\varphi$ be a trace on the discrete group $G$  with Fourier transform  $\mu_\varphi \in \mathrm{Prob}(\chars{G})$.  If the trace $\varphi$ is finite dimensional then   $\mathrm{supp}(\mu_\varphi)$ is finite. The converse direction holds true assuming that the group $G$ is virtually abelian. 
\end{lemma}

\begin{proof}
Let  $\varphi \in \traces{G}$  be a   finite dimensional trace with a corresponding trace representation  $(M,\pi,\tau)$. So $M$ is a finite dimensional von Neumann algebra. Then
\begin{equation}
M = \bigoplus_{i=1}^m \mathrm{M}_{d_i}(\CC), \quad \pi=\bigoplus_{i=1}^m \pi_i \quad \text{and} \quad \tau=\bigoplus_{i=1}^{m} \alpha_i\tau_i
\end{equation}
for some number $m \in \NN$, some    dimensions  $d_i \in \NN$, some finite dimensional unitary representations $\pi_i$ and some    real numbers  $\alpha_1,...,\alpha_m\in \left[0,1\right]$ satisfying $\sum_{i=1}^m\alpha_i=1$. Here  $\tau_i$ is the unique normalized trace $\frac{1}{d_i} \mathrm{tr}$ on the matrix algebra $\mathrm{M}_{d_i}(\CC)$.  In particular each function $\varphi_i=\tau_i\circ \pi_i$ is a character and  $\varphi=\sum_{i=1}^{m}\alpha_i \varphi_i$. We conclude  that the probability measure $\mu_\varphi$ is supported on the finite set $\{\varphi_1,...,\varphi_m \} \subset \chars{G}$.

For the converse direction,   assume that the group $G$ is virtually abelian and that $\mathrm{supp}(\mu_\varphi) =\{\varphi_1,...,\varphi_m\}\subseteq \Ch{G}$ for some $m\in\NN$. This means that  $\varphi=\sum_{i=1}^m \alpha_i \varphi_i$ for some real numbers $\alpha_i \in \left[0,1\right]$ satisfying $\sum_{i=1}^m\alpha_i =1$.

As the group $G$ is virtually abelian it has type $\mathrm{I}$ and all of its irreducible representations are finite dimensional. Let  $(M_i,\pi_i,\tau_i)$ be a trace representation corresponding to each character  $\varphi_i$. 
 Each  algebra  $M_i$ is  a type $\mathrm{I}$ factor, i.e. $M_i$ is isomorphic to the  algebra of all bounded operators on some Hilbert space \cite[Chapter 6]{bekka2020unitary}. With respect to this isomorphism  each representation $\pi_i$ is irreducible. Therefore each   $\pi_i$ is  finite dimensional so that each algebra $M_i$ is isomorphic to the matrix algebra $\mathrm{M}_{d_i}(\CC)$ for some dimension $d_i \in \NN$  and $\tau_i = \frac{1}{d_i} \mathrm{tr}$ is the unique normalized trace on $M_i$. The trace representation corresponding to the character $\varphi$ is given by $(M,\pi,\tau)$ where $\pi = \bigoplus_{i=1}^m \pi_i$, $M = \pi''(G) \le \bigoplus_{i=1}^m M_i$ and $\tau = \bigoplus_{i=1}^m \alpha_i \tau_i$. In particular the trace $\varphi$ is finite dimensional.
\end{proof}

We are  now able to deduce the necessary condition for Hilbert--Schmidt  stability from the introduction.

\begin{proof}[Proof of Proposition \ref{obs:intro: necc for HS}]
Let $N$ be an   abelian normal subgroup of the group $G$. Consider the dual action of $G$ on $\hat{N}$ by continuous automorphisms.  We will assume that the group $G$ is Hilbert--Schmidt stable and infer that the  topological dynamical system $(G,\hat{N})$ has dense periodic measures.

Let $\mu$ be any $G$-invariant Borel probability measure on the space $\chars{N} \cong \widehat{N}$. It gives rise to the relative trace   $\psi \in \tracerel{G}{N}$     by taking the Fourier transform of the measure $\mu$. Consider the trivial extension  $\varphi = \widetilde{\psi} \in \traces{G}$ as defined in  Equation (\ref{eq:extended trace}). 
Since  the group $G$ is Hilbert--Schmidt stable the trace $\varphi$ is a pointwise limit of a sequence $\varphi_n$ of finite dimensional traces,  see Corollary \ref{cor:Hadwin Shulman}.   Let $\mu_n$ denote the $G$-invariant Borel probability measure on the space $\Ch{N}$  corresponding to the restrictions $(\varphi_{n})_{|N}$. The measures $\mu_n$ have finite supports by  Lemma \ref{lem:f.d-trace-finitely-supported-measure}.

The barycenter map $\Prob{\Ch{N}}\to \Tr{N}$ is a homeomorphism in this case since the space of characters $\Ch{N} \cong \widehat{N}$ is compact. In particular   the probability measures $\mu_n$ converge to the measure $\mu$ in the weak-$*$ topology. We conclude that the dynamical  system $(G,\widehat{N})$ has dense periodic measures.
\end{proof}

\subsection*{Approximation of induced  finite dimensional traces}

The following approximation technique will play an important role in our applications towards Hilbert--Schmidt stability. Recall that a subgroup is   \emph{profinitely closed} if it is an intersection of finite index subgroups.

 \begin{prop}
 \label{lem:approximating finitely induced}
 Let $H$ be a subgroup of $G$ and $\psi \in \traces{H}$ be an almost $G$-invariant trace. Assume that  $\left[H:\ker \psi \right] < \infty$ and that $ \ker \psi$ is profinitely closed in $G$. Then the trace $ \Ind{G}{H}{\psi}$ is a limit of   finite dimensional traces on the group   $G$.
 \end{prop} 

\begin{proof}
To begin with  assume that the subgroup $H$ is normal in $G$ and that the trace $\psi$ is precisely $G$-invariant, i.e.  $\psi\in \relTr{G}{H}$. These additional assumptions imply that $\ker \psi $ is a normal subgroup of $G$. Up to replacing the group $G$ by its quotient $G / \ker \psi$, we will assume without further loss of generality  that  the trace $\psi$ is faithful, the subgroup $H$ is finite and  the group $G$ is residually finite.

Fix a descending   sequence $K_n$ of finite index normal subgroups of the group $G$ with trivial intersection. Up to discarding finitely many subgroups from this sequence one has  $K_n \cap H  = \{e\}$ for all $n \in \NN$. 

Consider the subgroups $G_n=K_n H$. The subgroup $G_n$ is a direct product of the two  subgroups $K_n$ and $H$ for all $n \in \NN$. Moreover $\left[G:G_n\right] < \infty$  for all $n \in \NN$. 
Consider the family of functions $\psi_n : G_n \to \CC$ given by
\begin{equation}
\psi_n(kh) = \psi(h) \quad \forall k \in K_n, \forall h \in H.
\end{equation}
Each function $\psi_n$ is the composition of the natural projection from $G_n$ to $H$ with the trace $\psi$ on the subgroup $H$. In particular each $\psi_n$ is   a finite dimensional trace   factoring though the finite group $H$.  Denote $\varphi_n =  \Ind{G}{G_n}{\psi_n}$. Each $\varphi_n$ is a finite dimensional trace on  the group $G$ by Lemma \ref{lem:induction-of-fd-is-fd} below. Observe that
\begin{equation}
\label{eq:limit 1}
 \varphi_n(g)  = \Ind{G}{G_n}\psi_n(g)=\widetilde{\psi}_n(g)=\psi_n(g)=\psi(g) \quad \forall n \in \NN
\end{equation}
for all elements $g \in H$ and that
\begin{equation}
 \label{eq:limit 2}
    \lim_{n\to \infty} \varphi_n(g) = \lim_{n\to\infty} \Ind{G}{G_n}{\psi_n}(g)=\lim_{n\to\infty} \widetilde{\psi}_n(g) = 0
\end{equation}
for all elements $g \in G \setminus H$.  Equations (\ref{eq:limit 1}) and (\ref{eq:limit 2}) put together imply the desired conclusion (given the additional assumptions).

In the general case consider the finite index subgroup $G_\psi  $ defined in Equation (\ref{eq:G phi}). The subgroup $G_\psi$ normalizes the subgroup $H$ and the trace $\psi$ is $G_\psi$-invariant. Relying on the previous paragraphs we find a sequence $\varphi_n$ of finite dimensional traces on the group $G_\psi$ converging to $\Ind{G_\psi}{H}{\psi} \in \traces{G_\psi}$. Induction in stages (Lemma \ref{lem:induction in stages}) combined with the continuity of induction (Lemma \ref{lem:continuity of induction}) gives
\begin{equation}
\label{eq:conclusion of induced}
    \Ind{G}{H}{\psi}=
    \Ind{G}{G_\psi}{\Ind{G_\psi}{H}{\psi}}=
    \Ind{G}{G_\psi}{\left(\lim_{n\to \infty}\varphi_n\right)}=
    \lim_{n\to \infty}\Ind{G}{G_\psi}\varphi_n.
\end{equation}
The functions on the right-hand side of Equation (\ref{eq:conclusion of induced}) are finite dimensional traces on the group $G$  by Lemma \ref{lem:induction-of-fd-is-fd} below. 
\end{proof}

Proposition \ref{lem:approximating finitely induced} is to be compared with \cite[Theorem 7]{hadwin2018stability} as well as with \cite[Proposition 8.1]{becker2019stability}.

\begin{lemma}\label{lem:induction-of-fd-is-fd}
Let $H$ be a finite index subgroup of the group $G$. If $\varphi \in \traces{H}$ is a finite dimensional trace on  the group $H$ then   $\Ind{G}{H}{\varphi} \in \traces{G}$ is a finite dimensional trace on the group $G$.
\end{lemma}

\begin{proof}
Let $\varphi$ be a finite dimensional trace on the group $H$. This means that $\varphi=\tau\circ\pi$ for some trace representation $(M,\pi,\tau)$ where the von Neumann algebra $M$ is finite dimensional. The induced representation $\Ind{G}{H}{\pi}$ of the group $G$ can be viewed as a representation  into the group of unitaries of the von Neumann algebra $\mathcal{M}_d(M)$ of $d$-by-$d$ matrices with entries in $M$   where $d=[G:H]$. By comparing the   definitions of induced traces  and induced representations  we see that 
\begin{equation}
   \Ind{G}{H}{\varphi} = \frac{1}{d}(\tr\otimes \tau) \circ\Ind{G}{H}{\pi}.
\end{equation}
It follows that the trace representation corresponding to $\Ind{G}{H}{\varphi}$ is finite dimensional. Therefore the trace $\Ind{G}{H}{\varphi}$ is finite dimensional.
\end{proof}

\section{Metabelian groups and Hilbert--Schmidt stability}
\label{sec:metabelian and stability}

Let $G$ be a finitely generated metabelian group. Assume that $N$ is an abelian normal subgroup of $G$ so that the quotient $G/N$ is abelian. We provide sufficient conditions for the group $G$ to be Hilbert--Schmidt stable.

\begin{theorem}
\label{thm:stability for metabelian}
If the topological dynamical system $(G,\widehat{H}^\mathrm{ab})$ has dense periodic measures  for every subgroup   $H$ with $N \le H \le G$ then the metabelian group $G$  is Hilbert--Schmidt stable.
\end{theorem}

The notation $H^\textrm{ab}$ stands for   abelianization, i.e. $H^\textrm{ab} = H / \left[H,H\right]$.

\begin{proof}[Proof of Theorem \ref{thm:stability for metabelian}]
Assume that the dynamical systems $(G,\widehat{H}^\textrm{ab})$ as   in the statement of the theorem   all have dense periodic measures. We will rely on the Hadwin--Shulman criterion by showing that any character of $G$   is a pointwise limit of finite dimensional traces on  the group $G$, see   Corollary \ref{cor:Hadwin Shulman}.

Let $\varphi$ be a character of the group $G$. The classification of characters of metabelian groups (see Theorem \ref{thm:characters of metabelian groups}) says that there exists a subgroup $H$ with   $N\leq H\leq G$ and a $G$-invariant  trace $\psi$ on the abelianization  $H^{\mathrm{ab}}$ such that  $\varphi=\Ind{G}{H}{\psi}$. Let $\mu$ be the $G$-invariant Borel probability measure on the Pontryagin dual $\widehat{H}^\mathrm{ab}$  corresponding to the trace $\psi$ via the Fourier transform. The dynamical system $(G,\widehat{H}^\textrm{ab})$ has dense periodic measures by assumption. Therefore there is a sequence $\mu_n$ of finitely supported $G$-invariant probability measures on $\widehat{H}^\textrm{ab}$ converging to the measure $\mu$ in the weak-$*$ topology. Let $\psi_n \in \traces{H^\textrm{ab}}$ be the Fourier transform of the measure $\mu_n$ for each $n$. The traces $\psi_n$ are $G$-invariant and satisfy $\lim_n \psi_n = \psi$ by the continuity of the Fourier transform over abelian groups. Furthermore each trace $\psi_n$ is finite dimensional   by   Lemma \ref{lem:f.d-trace-finitely-supported-measure}. 
Continuity of induction (Lemma \ref{lem:continuity of induction}) gives  
\begin{equation}
\lim_n \Ind{G}{H}{\psi_n} = \Ind{G}{H}{ \lim_n \psi_n} = \Ind{G}{H}{ \psi} = \varphi.
\end{equation}

Denote $\Gamma = G/H$ and regard $\Gamma$  as a subgroup of $\mathrm{Aut}(H^\ab)$. The semidirect product $\Gamma \ltimes H^\ab$ is finitely generated. In this situation Lemma \ref{lem:abelian-finite-index-kernel-G-invariant} allows us to assume without loss of generality that  $\left[H^\ab:\ker \psi_n\right] < \infty$ for all $n$.
Each    $\ker \psi_n$ is a normal subgroup of the finitely generated metabelian group $G$. As such each $\ker \psi_n$ is profinitely closed in the group $G$ by a classical theorem of Hall \cite{hall1959finiteness}.  We conclude that each induced trace $\Ind{G}{H}{\psi_n} $ is a limit of finite dimensional traces   by making use of   Proposition \ref{lem:approximating finitely induced}.
\end{proof}

Theorem \ref{Cor:metabelian-grps-HS-stable} of the introduction is a special case of the above Theorem \ref{thm:stability for metabelian}.

\subsection*{Hilbert--Schmidt stable metabelian groups}

We are ready to present the proofs of the theorems from the introduction \S\ref{sec:intro} dealing with metabelian groups.

\begin{prop}
\label{prop:Z-system}
Let $G = \ZZ \ltimes N$ be a finitely generated metabelian group. The topological dynamical system $(\ZZ, \widehat{N})$ has dense periodic measures if and only if the group $G$ is Hilbert--Schmidt stable.
\end{prop}
\begin{proof}
Assume that $(\ZZ,\widehat{N})$ has dense periodic measures. Let  $H$ be any fixed normal subgroup of $G$ satisfying $N \le H \le G$. It will suffice to show that the dynamical system $(G,\widehat{H}^\textrm{ab})$ has dense periodic measures according to   Theorem \ref{thm:stability for metabelian}. There are two separate cases to consider. If $H = N$ then $\widehat{H}^\textrm{ab} = \widehat{N}$ and the system in question has dense periodic measures by assumption. Otherwise $N \lneq H$, or equivalently $\left[G: H\right] < \infty$.  In that case the $G$-action on the space $\widehat{H}^\textrm{ab}$ factors through the finite quotient group $G/H$ and as such it certainly has dense periodic measures. The converse direction is an immediate consequence of Proposition \ref{obs:intro: necc for HS}.
\end{proof}

\begin{proof}[Proof of Corollary \ref{Cor:metabelian-grps-HS-stable}]
We show that the metabelian groups listed in the statement are Hilbert--Schmidt stable.

(1) Let  $M_k$ be the  free metabelian group of rank $k$. In other words   $M_k \cong F_k / F''_k$ where $F_k$ is the free group of rank $k$. Let $H$ be any     subgroup of the free metabelian group satisfying $M'_k \le H \le M_k$.
    To prove that the group $M_k$ is Hilbert--Schmidt stable it suffices by Theorem \ref{thm:stability for metabelian} to show that the dynamical system $(M_k,\widehat{H}^\ab)$ has dense periodic measures.
    Consider the subgroup $R$ of the free group $F_k$ satisfying  $F'_k \le R \le F_k$ and corresponding to the subgroup $H$ via the correspondence theorem. Since $F'_k \le R$ we have $F''_k \le R'$. In particular the subgroup $R' \le F_k$ corresponds to the subgroup $H' \le M_k$ via the correspondence theorem.  The third isomorphism theorem says that $    H/ H' \cong  R/ R'$.
    In other words $H^\ab$ and $R^\ab$ are isomorphic as  $M_k$-modules. Using the method of free differential calculus one can show that the  $M_k$-module $R^{\ab}$ is  isomorphic to a submodule of the $M_k$-module $\bigoplus_{i=1}^k  \ZZ \left[A\right] $ where $A \cong M_k / H \cong F_k / R$, see \cite[Chapter 11, Theorem 1]{johnson1997presentations} for details. 
The Bernoulli dynamical system $(M_k, \prod_A \TT^k)$ is a compact extension of the    dynamical system $(M_k,\widehat{H}^\ab)$.   Any Bernoulli dynamical system has dense periodic measures by   Proposition  \ref{prop:Bernoulli shift has dense periodic measures}. The same is   true for the dual dynamical system $(M_k,\widehat{H}^\ab)$ by  Item (\ref{cor:DPM-of-factors item:compact}) of Proposition \ref{prop:DPM-of-factors}.   
    
    (2) Consider the wreath product $G = A \wr \ZZ^d$ where $A$ is any finitely generated abelian group. Denote $N = \bigoplus_{\ZZ^d} A$ so that the normal subgroup $N$ as well as the quotient group $G/N$ are abelian. Let $H$ be any subgroup with $N \le H \le G$.    To prove that the group $G$ is Hilbert--Schmidt stable it suffices by Theorem \ref{thm:stability for metabelian} to show that the dynamical system $(G,\widehat{H}^\ab)$ has dense periodic measures.
     We claim that $\left[H,H\right] = \left[H,N\right]$. Indeed given any pair of elements $qn,rm \in G$ with $q,r \in \ZZ^d$ and $n,m \in N$ standard commutator identities give
     $$ \left[qn,rm\right] = \left[q,m\right] - \left[r,n \right].$$
     The claim follows. As the quotient $H/N$ is a free abelian group we obtain the direct sum  $H^\ab \cong   (H/N) \oplus N/\left[H,N\right]$. The dual dynamical system $(G,\widehat{H}^\ab)$ is isomorphic to the product system $(G,\widehat{H/N} \times X)$ where $\widehat{H/N}$ is a finite dimensional torus and $$X = \widehat{N}^H = \{x \in \widehat{N} \: : \: hx = x \;\; \forall h \in H\}.$$ The  $G$-action is trivial in the first factor and corresponds to the dual action in the second factor. Note that the topological dynamical system $(G/H,X)$ is  a Bernoulli system over a certain compact abelian group.  At this point we may conclude exactly as in case (1).
    
(3)    The Baumslag--Solitar  group  $\mathrm{BS}(1,n)$ is isomorphic to $ \ZZ \ltimes    \ZZ\left[1/n\right]$. The Pontryagin dual of the abelian group $\ZZ\left[1/n\right]$ is the $n$-adic solenoid $\mathrm{S}_n$ and the dual action corresponds to multiplication by $n$. 
 The dynamical system $(\ZZ,\mathrm{S}_n)$ has dense periodic measures by Proposition \ref{prop:p mult has dense periodic measures}. We conclude relying on Proposition \ref{prop:Z-system}.
 
 (4) Consider the group $G = \ZZ \ltimes_\alpha \ZZ^d$ where $\alpha \in \mathrm{GL}_d(\ZZ)$ is an ergodic torus automorphism. The topological dynamical system $(\ZZ,\TT^d)$     has dense periodic measures by  Theorem \ref{thm:marcus}. We conclude relying on Proposition \ref{prop:Z-system}.
\end{proof}

More generally, consider any  metabelian group of the form 
$G = \ZZ \ltimes _\alpha \ZZ\left[x,x^{-1}\right]/\left(p\right)$ where $ p \in \ZZ\left[x,x^{-1}\right]$ is not of the form $p(x,x^{-1}) = x^n c(x^m)$ for  some cyclotomic polynomial $c$ and some $n,m\in \ZZ$. Similar methods can be used to show that the group $G$ is Hilbert--Schmidt stable.
Indeed the corresponding topological dynamical system is expansive and has completely positive entropy \cite{schmidt2012dynamical}. Therefore  it has periodic specification \cite[Theorem 5.2]{lind1999homoclinic}. As such it  has  dense periodic measures.

\begin{proof}[Proof of Corollary \ref{Cor:compact-ring}]
Let $k$ be a non-Archimedean local field with ring of integers $\mathcal{O}$. Let $A$ be any infinite   subgroup of  $\mathcal{O}^*$. We will show that the group   $G = A \ltimes \widehat{\mathcal{O}}$ is not Hilbert--Schmidt stable by showing that the topological dynamical system $(A,\widehat{\mathcal{O}})$ does not have dense periodic measures, see Proposition \ref{obs:intro: necc for HS}. Indeed the singleton $\{0\}$ is the only finite $A$-orbit in the dual action on the Pontryagin dual group $\widehat{\mathcal{O}}$. In particular the Haar measure on $\widehat{\mathcal{O}}$ is certainly not a weak-$*$ limit of $G$-invariant probability measures of finite support.
\end{proof}

Let $G$ be a discrete group and  $X$   a compact abelian group admitting a $G$-action by continuous automorphisms.\label{definition of dcc}
This action is said to satisfy the \emph{descending chain condition} (d.c.c) if every descending chain of closed $G$-invariant subgroups of $X$ stabilizes. The topological dynamical system $(G,X)$ satisfies the descending chain condition   if and only if  the Pontryagin dual abelian group $\widehat{X}$ is a Noetherian module over the group ring $\ZZ\left[G\right]$. See   \cite{schmidt2012dynamical} for more on this property.

\begin{proof}[Proof of Theorem \ref{Thm:metabelian-grp-charactarization-HS-stability}]
$(\ref{item:all metabelian are stable}) \Rightarrow (\ref{item:all dcc are sofic})$. Assume that all finitely generated metabelian groups are Hilbert--Schmidt stable. 
Let $(G,X)$ be a topological dynamical system where $G$ is a finitely generated abelian group and $X$ a compact abelian group admitting a $G$-action by automorphisms    satisfying the d.c.c. Choose any finite index torsion-free subgroup  $G_0 \le G$. The metabelian group $G_0 \ltimes \widehat{X}$ is finitely generated and hence Hilbert--Schmidt stable by the assumption. Therefore the topological dynamical system $(G_0,X)$ has dense periodic measures by Proposition \ref{obs:intro: necc for HS}. The density of periodic measures for the system $(G,X)$ follows from Lemma \ref{lem:DPM-finite-index}.

$(\ref{item:all dcc are sofic}) \Rightarrow (\ref{item:all metabelian are stable})$. Assume that any $\ZZ^d$-action by automorphisms on any compact abelian group satisfying the d.c.c has dense periodic measures. Let $G$ be any finitely generated metabelian group. 
The assumption implies that the topological dynamical system  system $(G ,\widehat{H}^\ab)$ has dense periodic measures for any normal subgroup $H \lhd G$. It follows from    Theorem \ref{thm:stability for metabelian}  that the group $G$ is Hilbert--Schmidt stable. 
\end{proof}

\section{Polycyclic groups and Hilbert--Schmidt stability}
\label{sec:polycyclic and stability}

This last section deals with various questions concerning the Hilbert--Schmidt stability of virtually polycyclic groups. 

We start by developing   a sufficient condition for such groups to be stable in terms of  the density of periodic measures for certain associated dynamical systems, see Theorem \ref{thm:HS-stabilit-polycyclic-groups}. We proceed by constructing a family  a non-metabelian polycyclic   stable groups via a certain algebraic idea of rank-one modules, see Theorem \ref{thm:DPM of PVC}. Finally we explore the connection between the higher-rank measure rigidity conjecture and stability relying on some algebraic number theory. 

For the reader's convenience we  restate Theorem \ref{thm:intro-HS-stabilit-polycyclic-groups} of the introduction.

\begin{thm}\label{thm:HS-stabilit-polycyclic-groups}
Let $G$ be a virtually polycyclic group. Assume that for any finite index subgroup $H \le G$ and any quotient $H 	\twoheadrightarrow L$  the dual dynamical system 
\begin{equation}
\label{eq:systems that need to be sofic - repeat}
    (L,
    \chars{\mathrm{Z}( \vFit{L})})
\end{equation}
has dense periodic measures. Then the group $G$ is Hilbert--Schmidt stable. 
\end{thm}

\begin{proof}
Let $\varphi$ be any character of the group $G$. In light of the Hadwin--Shulman criterion formulated in  Corollary \ref{cor:Hadwin Shulman} it will suffice to show that the character $\varphi$ is a pointwise limit of finite dimensional traces on the group $G$.

According to our study of the character theory of virtually polycyclic groups (Theorem \ref{thm:characters of virPoly, refinement})   there exist a finite index subgroup $H \le G$,  a quotient  $L$ of the subgroup $H$    and an  $L$-invariant trace $\psi$ on the subquotient $F=\FC{\vFit{L}}$ such that  $\varphi=\Ind{G}{F}{\psi}$.

We claim that the  dynamical system $(L,\chars{F})$ has dense periodic measures. Let $L_0$ be any finite index normal subgroup of the group $L$ whose action on  every connected component of the character space $\chars{F}$   is  weakly algebraically isomorphic to the $L_0$-action on $\chars{\mathrm{Z}(\vFit{L})}$, see Proposition \ref{prop:action of dual of normal FC subgroup}.  The two subgroups $\mathrm{Z}(\vFit{L_0})$ and $ \mathrm{Z}(\vFit{L})$ are commensurable (i.e. their intersection   is of  finite index in both). Therefore the  $L_0$-actions on   $\chars{\mathrm{Z}(\vFit{L})}$ and on $\chars{\mathrm{Z}(\vFit{L_0})}$   are weakly algebraically isomorphic. The latter $L_0$-action has dense periodic measures  by the assumption. So  the dynamical system $(L_0,\chars{F})$ has dense periodic measures by the transitivity of weak algebraic isomorphisms and by statement (2) of Corollary \ref{cor:DPM-of-factors}. 
The claim follows from   Lemma  \ref{lem:DPM-finite-index}.

Let $\mu$ be the $L$-invariant probability measure on the character space $\chars{F}$ determined by the Fourier transform of the trace $\psi$.  The above claim says that  there is a sequence  of $L$-invariant probability measures $\mu_n \in \mathrm{Prob}(\chars{F})$ with finite supports that   converges   to the measure $\mu$ in the weak-$*$ topology. We may  assume without loss of generality that the corresponding traces $\psi_n=\widehat{\mu}_n \in \traces{F}$ satisfy $\left[F:\ker \psi_n\right] < \infty$ in light of   Lemma \ref{lem:finite index kernel for FC}. Denote $\varphi_n = \Ind{G}{F}{\psi_n}$ for all $n \in \NN$.


Recall that any subgroup of the virtually polycyclic group $H$ is profinitely closed \cite{malcev1951} (see also \cite[1.3.10]{lennox2004theory}). 
Therefore we may  apply  Proposition \ref{lem:approximating finitely induced} and deduce that each trace $\varphi_n$ is a pointwise limit of finite dimensional traces on the group $G$. As the sequence of traces $\psi_n  $ converges pointwise to the trace $\psi$, the sequence of traces $\varphi_n$ 
converges pointwise  to the trace $\varphi$   by  the continuity of induction (Lemma \ref{lem:continuity of induction}).  The desired conclusion follows.
\end{proof}

\begin{proof}[Proof of Corollary \ref{Thm:nilpotent are HS}]
Let $G$ be a finitely generated virtually nilpotent group. The group $G$ is in particular virtually polycyclic so that the stability criterion of  Theorem \ref{thm:HS-stabilit-polycyclic-groups} applies. Any subquotient $L$ of the group $G$ is virtually nilpotent. In particular   $L = \vFit{L}$. The topological dynamical system in question $(L, \widehat{\mathrm{Z}(L)}) $ certainly has dense periodic measures as the $L$-action on the character space of its center $\mathrm{Z}(L)$ is trivial. 
\end{proof}


\begin{proof}[Proof of Theorem \ref{Thm:poly-grp-charactarization-HS-stability}]
The two implications
$(\ref{item:vPoly})\Rightarrow (\ref{item:meta})$
and
$(\ref{item:amenable-sofic})\Rightarrow (\ref{item:abelian-sofic})$  are immediate.

$(\ref{item:meta})\Rightarrow (\ref{item:abelian-sofic})$. Let $G$ be an abelian subgroup of the group of automorphisms $\GL{d}{\ZZ}$ for some $d \in \NN$. The corresponding semidirect product $\widetilde{G}=G \ltimes \ZZ^d$ is Hilbert--Schmidt  stable by assumption.  Proposition \ref{obs:intro: necc for HS} implies that the dynamical system $(\widetilde{G},\TT^d)$ has dense periodic measures. Equivalently the dynamical system   $(G,\TT^d)$ has dense periodic measures, as required.

$(\ref{item:vPoly})\Rightarrow (\ref{item:amenable-sofic})$. Let $G$ an amenable subgroup of the group of automorphisms $\GL{d}{\ZZ}$ for some $d \in \NN$. The linear group $G$ is virtually solvable by the Tits alternative \cite{tits1972free}. As such the   group $G$ is virtually polycyclic by \cite{malcev1951}, see also \cite[\S3.2]{lennox2004theory}. Therefore the semidirect product  $\widetilde{G}=G \ltimes \ZZ^d$ is   Hilbert--Schmidt stable by assumption. We conclude exactly as in the previous paragraph.

$(\ref{item:abelian-sofic})\Rightarrow (\ref{item:vPoly})$. Let $G$ be a virtually polycyclic group. We will rely on Theorem \ref{thm:HS-stabilit-polycyclic-groups}   to  deduce that the group $G$ is   Hilbert--Schmidt stable.

With this goal in mind, consider the topological dynamical system $(L,\widehat{Z})$ where 
  $L$ is a quotient of some finite index subgroup  of the group $G$ and
  $Z$ is the subquotient $\mathrm{Z}(\vFit{L})$. The $L$-action on the space $\widehat{Z}$ factors through some virtually abelian quotient. Indeed  the action of the subquotient $\vFit{L}$ on the space $\widehat{Z}$ is trivial   and the quotient $L/\vFit{L}$ is virtually abelian.   The  dynamical system $(L, \widehat{Z})$ has dense periodic measures by assumption. This completes the proof.
\end{proof}

\subsection*{Rank one type   and upper triangular matrices}
\label{sec:semisimple actions}

We study a class of polycyclic non-metabelian  groups for which we are able   to show   Hilbert--Schmidt stability relying  on the methods of this work. In particular we establish Theorem \ref{Thm:upper triangular} of the introduction.

Fix a  finite dimensional rational vector space $V$. Let $A$ be an abelian subgroup of $\mathrm{GL}(V)$ generated by semisimple transformations. It follows that   every element of the group $A$ is semisimple \cite[\S 4]{borel2012linear}. 

Consider the $\QQ$-algebra $ \mathcal{A} = \QQ A$ spanned by the group $A$ inside the $\QQ$-algebra $\mathrm{End}_\QQ(V)$. As    $\mathcal{A}$ is diagonizable it has no non-zero nilpotent elements. Moreover   $\mathcal{A}$ is  commutative and Artinian.  These conditions imply that  the $\QQ$-algebra $ \mathcal{A}$ is semisimple \cite[Corollary 2.5]{farb2012noncommutative}.

\begin{definition*}
A simple $\mathcal{A}$-module $M$ has \emph{rank one} if the restricted action of the group $A$ on $M$  factors through some homomorphism $f : A \to \ZZ$. An   $\mathcal{A}$-module $M$ has \emph{rank one type} if every simple $\mathcal{A}$-submodule of $M$
 has rank one. The group $A$ itself has \emph{rank one type} if the $\mathcal{A}$-module $V$ has rank one type.
 \end{definition*}

\begin{prop}
\label{prop:subquotient of PVC is PVC}
Let $M$ be an $\mathcal{A}$-module of rank one type. Then any $\mathcal{A}$-submodule $N \le M$ as well as any quotient $\mathcal{A}$-module $M \twoheadrightarrow L \cong M/N$ has rank one type.
\end{prop}
\begin{proof}
The $\mathcal{A}$-module $M$ has a  unique decomposition   $M = \bigoplus_{i=1}^l M_i$ into its isotypic components $M_1,\ldots,M_l$.  Any $\mathcal{A}$-submodule $N$ of $M$ must be of the form $N = \bigoplus_{i=1}^l N_i$ for some $\mathcal{A}$-submodules  $0 \le N_i \le M_i$. It follows from this description that the $\mathcal{A}$-submodule $N$ as well as the quotient $\mathcal{A}$-module 
\begin{equation}
L \cong M/N \cong \bigoplus_{i=1}^l M_i/N_i    
\end{equation}
have rank one type.
\end{proof}

\begin{lemma}
\label{lemma:PVC has dense periodic measures}
Let $A$ be a commutative subgroup of   
$ \mathrm{GL}_d(\ZZ)$ generated by semisimple transformations. If the group $A$ has rank one type then   the dynamical system $(A, \TT^d)$ has dense periodic measures.
\end{lemma}
\begin{proof}
Consider the rational vector space $V = \QQ^d$ regarded as a module over the semisimple $\QQ$-algebra $\mathcal{A} = \QQ A$. There is a uniquely determined decomposition $V = \bigoplus_{i=1}^l V_i$ into isotypic components for some $l \in \NN$. Each $\mathcal{A}$-submodule $V_i$ is  in particular   a rational subspace of the rational vector space $V$.

There is a $d$-dimensional torus $T$ admitting a   product decomposition $T = \prod_{i=1}^l T_i$ into $A$-invariant factors  as well as an $A$-equivariant finite-sheeted covering map $T \to \TT^d$   such that each sub-torus $T_i$ covers the image of the rational subspace $V_i$ inside the torus $T$.

The action of the group $A$ on each torus $T_i$ factors through a cyclic group $\ZZ$ generated by a single semisimple transformation $A_i \in \mathrm{Aut}(T_i)$.
Up to passing to a finite index subgroup of the group $A$ and reindexing, we may assume that each transformation $A_i$ is ergodic on its respective torus $T_i$ with respect to the Haar measure, expect possibly for 
a single index $i_0$ for which     $A_{i_0}$ is the identity transformation.

We conclude that the dynamical system $(A,T)$ has dense periodic measures by   relying on Theorem \ref{thm:marcus}.
The density of periodic measures for the dynamical system $(A,\TT^d)$ follows from Corollary \ref{cor:DPM-of-factors} concerning finite factors.
\end{proof}

Let $N$ be a finitely generated torsion-free nilpotent group. Recall that there is a well-defined   rational Lie algebra $\mathcal{L} = \mathrm{Lie}(N)$ associated to the group $N$ as well as an $\mathrm{Aut}(N)$-equivariant bijection $\log : N \to \mathcal{N}$ with inverse $\exp : \mathcal{N} \to N$, see \cite[\S 6]{segal2005polycyclic}.

\begin{theorem}
\label{thm:DPM of PVC}
Let $A$ be an abelian subgroup of $\mathrm{Aut}(N)$. If
$A$ is generated by semisimple transformations and has rank one type when regarded as a subgroup of $\mathrm{GL}(\mathcal{L})$ then the polycyclic group $G = A \ltimes N$ is Hilbert--Schmidt stable.
\end{theorem}
\begin{proof}
Denote $\mathcal{A} = \QQ A$ so that $\mathcal{A}$ is a  semisimple $\QQ$-algebra. The assumption of the theorem says that the rational Lie algebra $\mathcal{N}$ is a rank one type  $\mathcal{A}$-module.

Fix an arbitrary finite index subgroup $H \le G$ and a normal subgroup $K \lhd H$. Denote $\overline{H} = H/K$ and $\overline{Z} = \mathrm{Z}(\vFit{\overline{H}})$. We will conclude  that the group $G$ is Hilbert--Schmidt stable     by relying on Theorem
\ref{thm:HS-stabilit-polycyclic-groups} and   showing that the topological dynamical system $(\overline{H}, \chars{\overline{Z}})$ has dense periodic measures.

The subquotient $\overline{Z}$ is a finitely generated abelian group. Therefore  $\overline{Z}$   admits some  \emph{torsion-free} characteristic subgroup $\overline{C}$ satisfying  $\left[\overline{Z}:\overline{C}\right] < \infty$. 
Consider the subquotient
\begin{equation}
\overline{M} = \overline{C} \cap (NK \cap H)/K.
\end{equation}

 Let $\mathcal{C} = \mathrm{Lie}(\overline{C})$ and $\mathcal{M} = \mathrm{Lie}(\overline{M})$  be   the rational Lie algebras  of the finitely generated torsion-free abelian subquotients $\overline{C}$ and $\overline{M}$ respectively. Clearly $\mathcal{M} \le \mathcal{C}$. Let $A_0 = \mathrm{N}_{A}(H) \cap \mathrm{N}_{A}(K)$ so that  $A_0$  is a   finite index subgroup of the abelian group $A$. In particular   $A_0$ spans the $\QQ$-algebra $\mathcal{A}$, i.e.  $\mathcal{A} = \QQ A_0$.
This   implies that  both $\mathcal{M}$ and $\mathcal{C}$ are $\mathcal{A}$-modules. 
 The fact that   the $\QQ$-algebra $\mathcal{A}$ is semisimple implies that the $\mathcal{A}$-submodule $\mathcal{M}$ admits a complement in $\mathcal{C}$, i.e.  there exists some  $\mathcal{A}$-submodule $\mathcal{N} \le \mathcal{C}$  satisfying 
$\mathcal{C} = \mathcal{M} \oplus \mathcal{N}$. 

We claim that   the action of the subgroup $A_0$ on the module $\mathcal{N}$ obtained by restricting the action of the algebra $\mathcal{A}$ is trivial.  To see this consider any element  $x \in \mathcal{C}$ with $\exp x = g \in \overline{C}$. Write   $  g = anK$ for some pair of elements   $a \in A, n \in N$. Any element $\alpha \in A_0 $ satisfies
\begin{equation}
g^{-1} \alpha(g)  = n^{-1} a^{-1} \alpha(an) K = n ^{-1} a^{-1}a \alpha(n) K = n^{-1} \alpha(n) K.
\end{equation}
It follows that
\begin{equation}
\exp(x)^{-1} \alpha(\exp(x)) \in \overline{M}.    
\end{equation}
The $A_0$-equivariance of the exponential map gives $\alpha(x) -  x \in \mathcal{M}$ for all elements $x\in \mathcal{C}$ and $\alpha \in A_0$. Therefore $\alpha(y) - y \in \mathcal{M} \cap \mathcal{N} = \{0\}$ for all elements $y \in \mathcal{N}$. The above claim follows.



The $\mathcal{A}$-module $\mathcal{M}$ is a subquotient of the   $\mathcal{A}$-module $\mathcal{L}$ and as such has rank one type by Proposition \ref{prop:subquotient of PVC is PVC}. The $\mathcal{A}$-module $\mathcal{N}$   certainly has rank one type for the restricted action of the subgroup $A_0$ on it is trivial. Therefore the  $\mathcal{A}$-module $\mathcal{C} = \mathcal{M} \oplus \mathcal{N}$ has rank one type as well.

The $\overline{H}$-actions on the two dual spaces $\chars{C}$ and $\chars{Z}$ are weakly algebraically isomorphic (in the sense discussed in \S\ref{sec:fc induced}). 
Since the $\mathcal{A}$-module $\mathcal{C}$ has rank one type the $\overline{H}$-action on the space $\chars{\overline{Z}}$   has dense periodic measures according to the two Lemmas \ref{lem:DPM-finite-index} and \ref{lemma:PVC has dense periodic measures}.
\end{proof}

We are ready to prove that upper triangular groups over certain rings of algebraic integers are Hilbert--Schmidt stable.

\begin{proof}[Proof of Theorem \ref{Thm:upper triangular}]
Let $\mathcal{O}$ be the ring of algebraic integers in some number field. Assume   that the group of units $\mathcal{O}^*$ satisfies  $\mathrm{rank}( \mathcal{O}^*) = 1$. 
Let $A$ and $N$ respectively denote the diagonal and upper triangular subgroups of the matrix group $\mathrm{GL}_d(\mathcal{O}) $ for some fixed  $d\in \NN$.  The group $N$ is finitely generated, nilpotent and torsion free (see \cite[1.2.20]{lennox2004theory}). The assumption on the group of units $\mathcal{O}^*$ implies that  the action of the diagonal group $A$  on the rational Lie algebra $\mathcal{N} = \mathrm{Lie}(N)$ has rank one type.  The fact that the polycyclic group $G$ is Hilbert--Schmidt stable follows from Theorem  \ref{thm:DPM of PVC}.
\end{proof}

\subsection*{The higher-rank measure rigidity conjecture} 

We discuss the relationship between the higher-rank measure rigidity conjecture and stability. 


\begin{prop}
\label{prop:properties of semi direct product fields}
Let $\mathcal{O}$ be the ring of algebraic integers in the number field $k$. Assume that $k$ is totally real   and that $\mathrm{rank}(\mathcal{O}^*) \ge 2$. Then the dual action of group of units $\mathcal{O}^*$ on the dual torus $\widehat{\mathcal{O}}$ is almost minimal, i.e. there are no invariant    proper closed infinite subsets.
\end{prop}
\begin{proof}
Consider the dual action of the group of units  $\mathcal{O}^*$ on the dual torus $\widehat{\mathcal{O}}$. 
According to \cite{berend1983multi} this action is almost minimal provided that  the following three conditions are satisfied:
\begin{enumerate}
    \item There is a unit     $u \in \mathcal{O}^*$ such that multiplication by the element $u^n$ is \emph{rationally irreducible}, i.e.  admits  no non-trivial invariant rational subspace in $k$ for all $n \in \NN$ (see \cite[Proposition 2.2.6]{katok2011rigidity} for several equivalent formulations of this statement).
    \item There is a unit $u \in \mathcal{O}^*$ such that multiplication by the element $u$ is represented by a hyperbolic matrix.
    \item There is a pair of units $u,v \in \mathcal{O}^*$ such that $u^n \neq v^m$ for all non-zero $n,m \in \NN$.
\end{enumerate}

In light of \cite[\S11, Lemma 2]{segal2005polycyclic} the dual action of a given unit $u \in \mathcal{O}^*$ is totally  irreducible if and only if $k = \QQ(u)$. Note that   there are finitely many   intermediate number fields $ \QQ \subset l \subset k$. On the other hand, since $k$ is totally real and  is in particular not a CM-field, the group of units $\mathcal{O}^*$   satisfies   $\mathrm{rank}(\mathcal{O}^* \cap l) < \mathrm{rank}(\mathcal{O}^*)$ for each proper number field $l \subsetneq k$, see
 \cite{parry1975units}. This shows that it is possible to choose some unit $u\in\mathcal{O}^*$ as required in Statement (1).

Statement (2) says that there is some unit $u \in \mathcal{O}^*$ whose dual action on the torus $\widehat{\mathcal{O}}$ is represented by a hyperbolic matrix, i.e. all   eigenvalues of the corresponding matrix are \emph{not} unimodular. These eigenvalues are the algebraic conjugates of $u$, i.e. the roots of the minimal polynomial of $u$. Each     eigenvalue has multiplicity  $\left[k:\QQ(u)\right]$. See \cite[VI,\S5]{lang2012algebra}. Since the number field $k$ is totally real any infinite order unit has the required property.

Statement (3) follows from the fact that $\mathrm{rank}(\mathcal{O}^*) \ge 2$. 
This concludes the proof.
\end{proof}

We are ready to complete the proof of the stability  result stated in the introduction conditional on the measure rigidity  conjecture.

\begin{proof}[Proof of Proposition \ref{prop:application of conjecture to stability}]
Let $\mathcal{O}$ be the ring of algebraic integers in some totally real number field $k$. We wish to show that  the metabelian polycyclic group $G = \mathcal{O}^* \ltimes \mathcal{O}$ is Hilbert--Schmidt stable.
If $\mathrm{rank}(\mathcal{O}^*) = 1$ then  the group
$G$ is Hilbert--Schmidt stable by Theorem \ref{thm:DPM of PVC}. Therefore we will assume     that  $\mathrm{rank}(\mathcal{O}^*) \ge 2$ from now on. 

The conditions of Proposition \ref{prop:properties of semi direct product fields} are satisfied so that the dual action of group of units $\mathcal{O}^*$ on the dual torus $\widehat{\mathcal{O}}$ is almost minimal.  By conditionally relying   on the higher rank measure rigidity conjecture \cite{margulis2000problems} (see the discussion in \S\ref{sec:intro}) we may assume   that any non-atomic $\mathcal{O}^*$-invariant  Borel probability measure on the dual group $\widehat{\mathcal{O}}$ is   the Haar measure.

Let $H$ be some normal subgroup of the group $G$ satisfying $
\mathcal{O} \le H \le G$ and
consider its abelianization $H^\textrm{ab}$. The fact that the action of $\mathcal{O}^*$ is irreducible implies that the derived subgroup $\left[H,H\right]$ is either equal to the normal subgroup $\mathcal{O} \lhd G$ or is trivial.
 In the first case the $G$-action on the abelianization $H^\textrm{ab}$ is trivial.  In the second case we obtain   that $\widehat{H}^\textrm{ab} = \widehat{\mathcal{O}}$. This completes the proof by relying on the higher rank measure rigidity conjecture combined with the criterion given in Theorem  \ref{thm:stability for metabelian}.
\end{proof}

\bibliographystyle{alpha}
\bibliography{stability}

\end{document}